\documentclass[12pt,reqno]{amsart}

\newcommand\version{July 1, 2017}

\usepackage{amsmath,amsfonts,amsthm,amssymb,amsxtra}

\newtheorem{theorem}{Theorem}
\newtheorem{proposition}[theorem]{Proposition}
\newtheorem{lemma}[theorem]{Lemma}

\theoremstyle{definition}

\theoremstyle{remark}

\renewcommand{\epsilon}{\varepsilon}

\renewcommand{\phi}{\varphi}
\newcommand{\R}{\mathbb{R}}

\DeclareMathOperator{\dist}{dist}

\DeclareMathOperator{\supp}{supp}

\newcommand{\ud}{\mathrm{d}}
\newcommand{\be}{\begin{equation}}      
\newcommand{\ee}{\end{equation}}
\newcommand{\lda}{\lambda}
\newcommand{\om}{\Omega}
\newcommand{\va}{\varepsilon}

\begin{document}

\title[Minimizers for the fractional Sobolev inequality --- \version]{Minimizers for the fractional Sobolev inequality\\ on domains}

\author{Rupert L. Frank}
\address[R. Frank]{Mathematisches Institut, Ludwig-Maximilans Universit\"at M\"unchen, Theresienstr. 39, 80333 M\"unchen, Germany, and Department of Mathematics, California Institute of Technology, Pasadena, CA 91125, USA}
\email{frank@math.lmu.de, rlfrank@caltech.edu}

\author{Tianling Jin}
\address[T. Jin]{Department of Mathematics, The Hong Kong University of Science and Technology, Clear Water Bay, Kowloon, Hong Kong}
\email{tianlingjin@ust.hk}

\author{Jingang Xiong}
\address[J. Xiong]{School of Mathematical Sciences, Beijing Normal University, Beijing 100875, China}
\email{jx@bnu.edu.cn}

\begin{abstract}
We consider a version of the fractional Sobolev inequality in domains and study whether the best constant in this inequality is attained. For the half-space and a large class of bounded domains we show that a minimizer exists, which is in contrast to the classical Sobolev inequalities in domains.
\end{abstract}

\maketitle

\renewcommand{\thefootnote}{${}$} \footnotetext{\copyright\, 2017 by
  the authors. This paper may be reproduced, in its entirety, for
  non-commercial purposes.}

\section{Introduction}

The fractional Sobolev inequality in $\R^n$ of order $\sigma\in (0,1)$ (with the additional assumption $\sigma<1/2$ if $n=1$) states that any function in $\mathring H^\sigma(\R^n)$ belongs to $L^\frac{2n}{n-2\sigma}(\R^n)$ and its norm in this space is controlled by its norm in $\mathring H^\sigma(\R^n)$. Here $\mathring H^\sigma(\R^n)$ denotes the space of all (real-valued) functions $u$ on $\R^n$ such that
$$
I_{n,\sigma,\R^n}[u] := \iint_{\R^n\times\R^n} \frac{(u(x)-u(y))^2}{|x-y|^{n+2\sigma}} \,\ud x\,\ud y
$$
is finite and such that $|\{ |u|>\tau\}|<\infty$ for all $\tau>0$. The best constant in this fractional Sobolev inequality, that is,
$$
S_{n,\sigma}(\R^n) := \inf_{0\not\equiv u \in\mathring H^\sigma(\R^n)} \frac{I_{n,\sigma,\R^n}[u]}{\left( \int_{\R^n} |u|^\frac{2n}{n-2\sigma}\,\ud x \right)^\frac{n-2\sigma}{n}} \,,
$$
was found by Lieb \cite{Lieb}, who also showed that this infimum is attained exactly by multiples, translates and dilates of the function $(1+|x|^2)^{-\frac{n-2\sigma}{2}}$. In fact, Lieb considered the dual version of the inequality, known as Hardy--Littlewood--Sobolev inequality, and proved a more general result. Alternative proofs of this result were later given in \cite{CaLo,FrLi1,FrLi2,FrLi3}. Lions \cite{Lions} proved that any normalized minimizing sequence for the optimization problem $S_{n,\sigma}(\R^n)$ is relatively compact up to translations and dilations. Interestingly, although not really relevant for us here, one can show \cite{ChLiOu,Li} that translates and dilates of the function $(1+|x|^2)^{-\frac{n-2\sigma}{2}}$ are the only positive solutions of the Euler--Lagrange equation corresponding to the minimization problem.

In this paper we are interested in the fractional Sobolev inequality on the half-space
$$
\R^n_+ = \{ (x',x_n)\in \R^{n-1}\times\R :\ x_n >0\}
$$
for functions vanishing on the boundary. Since the quadratic form $I_{n,\sigma,\R^n}$ is non-local, there are at least two natural ways of formulating such an inequality. The first one is to consider the minimization problem
$$
S_{n,\sigma}'(\R^n_+) := \inf\left\{ \frac{I_{n,\sigma,\R^n}[u]}{\left( \int_{\R^n} |u|^\frac{2n}{n-2\sigma}\,\ud x \right)^\frac{n-2\sigma}{n}} :\ 0\not\equiv u \in \mathring H^\sigma(\R^n) \,,\ \supp u \subset\R^n_+\right\}.
$$
Then, clearly, $S_{n,\sigma}'(\R^n_+) \geq S_{n,\sigma}(\R^n)$ and, in fact, using the dilation or translation invariance of the whole space problem, it is not difficult to see that
$$
S_{n,\sigma}'(\R^n_+) = S_{n,\sigma}(\R^n) \,.
$$
Moreover, by Lieb's classification result mentioned above, minimizers for $S_{n,\sigma}(\R^n)$ do not vanish on a half-space and therefore the infimum $S_{n,\sigma}'(\R^n_+)$ is not attained.

The second and more interesting way of formulating the problem consists in replacing $I_{n,\sigma,\R^n}$ by
$$
I_{n,\sigma,\R^n_+}[u] := \iint_{\R^n_+\times\R^n_+} \frac{(u(x)-u(y))^2}{|x-y|^{n+2\sigma}} \,\ud x\,\ud y
$$
and to define $\mathring H^\sigma(\R^n_+)$ as the completion of $C^1_c(\R^n_+)$ with respect to the quadratic form $I_{n,\sigma,\R^n_+}$. Then it is easy to see \cite[Lemma 2.1]{DyFr} that, assuming again $0<\sigma<1/2$ for $n=1$ and $0<\sigma<1$ for $n\geq 2$, any function in $\mathring H^\sigma(\R^n_+)$ belongs to $L^\frac{2n}{n-2\sigma}(\R^n_+)$ and its norm in this space is controlled by its norm in $\mathring H^\sigma(\R^n_+)$. One is naturally led to the minimization problem
$$
S_{n,\sigma}(\R^n_+) := \inf_{0\not\equiv u \in\mathring H^\sigma(\R^n_+)} \frac{I_{n,\sigma,\R^n_+}[u]}{\left( \int_{\R^n_+} |u|^\frac{2n}{n-2\sigma}\,\ud x \right)^\frac{n-2\sigma}{n}} \,.
$$
The following two theorems, which are our main results, show that under the condition $n\geq 4\sigma$ the minimization problem $S_{n,\sigma}(\R^n_+)$ behaves completely differently from the minimization problem $S_{n,\sigma}'(\R^n_+)$.

\begin{theorem}\label{mainhsbinding}
Let $0<\sigma<1/2$ if $n=1$ and $0<\sigma<1$ if $n\geq 2$ and assume that
\begin{equation*}
n\geq 4\sigma \,.
\end{equation*}
Then
\begin{equation}
\label{eq:bindingineq}
S_{n,\sigma}(\R^n_+) < S_{n,\sigma}(\R^n) \,.
\end{equation}
\end{theorem}

The second theorem says that under assumption \eqref{eq:bindingineq} the half-space analogues of the theorems of Lieb and Lions mentioned above hold.

\begin{theorem}\label{mainhs}
Let $0<\sigma<1/2$ if $n=1$ and $\sigma\in(0,1/2)\cup(1/2,1)$ if $n\geq 2$ and assume \eqref{eq:bindingineq}. Then any minimizing sequence for $S_{n,\sigma}(\R^n_+)$, normalized in $\mathring H^\sigma(\R^n_+)$, is relatively compact in $\mathring H^\sigma(\R^n_+)$, up to translations parallel to the boundary and dilations. In particular, the infimum is attained.
\end{theorem}

As we will see in the proof, assumption \eqref{eq:bindingineq} is not only sufficient, but also necessary for the relative compactness modulo symmetries of all minimizing sequences.

We do not know whether the assumption $\sigma\neq 1/2$ for $n\geq 2$ is necessary for the conclusion of Theorem \ref{mainhs}. In our proof this assumption allows us to use the fractional Hardy inequality in half-spaces \cite{BoDy} (see also \cite{FrSe}).

Not only does the minimization problem $S_{n,\sigma}(\R^n_+)$ behave differently from $S_{n,\sigma}'(\R^n_+)$, it also behaves differently from its local analogue. Namely, one has both
$$
\lim_{\sigma\to 1-} (1-\sigma) I_{n,\sigma,\R^n}[u] = c_n \int_{\R^n} |\nabla u|^2\,\ud x
\qquad\text{for all}\ u\in H^1(\R^n) \,,
$$
and
$$
\lim_{\sigma\to 1-} (1-\sigma) I_{n,\sigma,\R^n_+}[u] = c_n \int_{\R^n_+} |\nabla u|^2\,\ud x
\qquad\text{for all}\ u\in H^1(\R^n_+)
$$
for some explicit constant $c_n\in (0,\infty)$. (This is essentially contained in \cite{BoBrMi}.) Therefore both minimization problems $S_{n,\sigma}'(\R^n_+)$ and $S_{n,\sigma}(\R^n_+)$ can be seen as fractional analogues of the minimization problem
\begin{equation}
\label{eq:snloc}
S_{n}(\R^n_+) = \inf_{0\not\equiv u \in\mathring H^1(\R^n_+)} \frac{\int_{\R^n_+} |\nabla u|^2\,\ud x}{\left( \int_{\R^n_+} |u|^\frac{2n}{n-2}\,\ud x \right)^\frac{n-2}{n}}
\end{equation}
for $n\geq 3$. For the latter problem, however, we obtain by the same arguments as for the $S_{n,\sigma}'(\R^n_+)$ problem that $S_{n}(\R^n_+) = S_{n}(\R^n)$ (the latter being defined in an obvious way with integrals extended over all of $\R^n$ and allowing for functions in $\mathring H^1(\R^n)$) and that the infimum is not attained.

The discrepancy between the $S_{n,\sigma}(\R^n_+)$ and $S_{n,\sigma}'(\R^n_+)$ problems can be explained as a Br\'ezis--Nirenberg effect. For $u\in\mathring H^\sigma(\R^n_+)$ we write
$$
I_{n,\sigma,\R^n_+}[u] = I_{n,\sigma,\R^n}[u] - 2 \int_{\R^n_+} |u|^2 \int_{\R^n_-} \frac{\ud y}{|x-y|^{n+2\sigma}} \,\ud x = I_{n,\sigma,\R^n}[u] - \kappa_{n,\sigma} \int_{\R^n_+} \frac{|u|^2}{x_n^{2\sigma}} \,\ud x
$$
with a constant $\kappa_{n,\sigma}\in(0,\infty)$ whose precise value is not important for us. Therefore the $S_{n,\sigma}(\R^n_+)$ problem is the $S_{n,\sigma}'(\R^n_+)$ problem with an additional negative term, and it is this term that for $n\geq 4\sigma$ lowers the value of the infimum and produces a minimizer. The fact that a `lower order term' can produce these phenomena in high enough dimensions was observed by Br\'ezis and Nirenberg (motivated by work of Aubin \cite{Au}) and our Theorems \ref{mainhsbinding} and \ref{mainhs} are the analogues of the results of Br\'ezis--Nirenberg and Lieb in \cite{BrNi}. We mention also that fractional versions of the Br\'ezis--Nirenberg problem were studied in \cite{CaTa,SeVa}. Our problem is significantly more complicated than the traditional or fractional Br\'ezis--Nirenberg problems since the term $\int x_n^{-2\sigma} |u|^2\,\ud x$ scales in the same way as $I_{n,\sigma,\R^n_+}[u]$ and therefore is almost of the same strength.

\medskip

As an application of our Sobolev inequality on half-spaces we now consider the case of an arbitrary domain $\Omega\subset\R^n$. We put
$$
I_{n,\sigma,\Omega}[u] := \iint_{\Omega\times\Omega} \frac{(u(x)-u(y))^2}{|x-y|^{n+2\sigma}} \,\ud x\,\ud y
$$
and denote by $\mathring H^\sigma(\Omega)$ the completion of $C^1_c(\Omega)$ with respect to the non-negative quadratic form $I_{n,\sigma,\Omega}$. Let
$$
S_{n,\sigma}(\Omega) := \inf_{0\not\equiv u \in\mathring H^\sigma(\Omega)} \frac{I_{n,\sigma,\Omega}[u]}{\left( \int_{\Omega} |u|^\frac{2n}{n-2\sigma}\,\ud x \right)^\frac{n-2\sigma}{n}} \,.
$$
(Strictly speaking, $\mathring H^\sigma(\Omega)$ may or may not be a space of functions and in the definition of $S_{n,\sigma}(\Omega)$ one should minimize over functions in $C^1_c(\Omega)$. When $S_{n,\sigma}(\Omega)$, defined in this way, is positive, which is the case we are mostly interested in, then $\mathring H^\sigma(\Omega)$ is a space of functions and the above definition is equivalent.)

Let us recall some results about the validity of the Sobolev inequality on $\Omega$. For $n\geq 2$ and $\sigma>1/2$ one has $S_{n,\sigma}(\Omega)>0$ for any open set $\Omega$. This follows from \cite{DyFr}, which even shows that $\underline S_{n,\sigma} := \inf_\Omega S_{n,\sigma}(\Omega)>0$. In passing we mention that it is an open problem to compute $\underline S_{n,\sigma}$ and to analyze minimizing sequences of sets $\Omega$. On the other hand, when $n\geq 1$ and $\sigma<1/2$, one has $S_{n,\sigma}(\Omega)=0$ for any open set $\Omega$ of finite measure with sufficiently regular boundary; see Lemma \ref{nosob}. However, one does have $S_{n,\sigma}(\Omega)>0$ for $n\geq 1$ and $\sigma<1/2$ if $\Omega$ is the complement of the closure of a bounded Lipschitz domain or a domain above the graph of a Lipschitz function. This follows from the Sobolev inequality on $\R^n$ and the Hardy inequality from \cite{Dy}. The case $\sigma=1/2$ seems to be not really understood.

Our next result compares $S_{n,\sigma}(\Omega)$ with $S_{n,\sigma}(\R^n_+)$ for a class a open sets whose boundary has a flat part. It would be interesting to extend this result to a larger class of sets.

\begin{theorem}\label{mainbddbinding}
Let $n\geq 2$ and $1/2<\sigma<1$ and assume that $S_{n,\sigma}(\R^n_+)$ is attained. Let $\Omega\subset\R^n_+$ be an open set such that for some $\epsilon>0$ and some $a\in\partial\R^n_+$,
$$
B_\epsilon^+(a)\subset\Omega
$$
and such that $\R^n_+\setminus\Omega$ has non-empty interior. Then
\begin{equation}
\label{eq:mainbddass}
S_{n,\sigma}(\Omega) < S_{n,\sigma}(\R^n_+) \,.
\end{equation}
\end{theorem}

We recall that by Theorems \ref{mainhsbinding} and \ref{mainhs} the assumption $n\geq 4\sigma$ guarantees that $S_{n,\sigma}(\R^n_+)$ is attained. The reason for the assumption $\sigma>1/2$ will be explained after Proposition~\ref{prop:extremalfunct}.

Finally, we show that the strict inequality \eqref{eq:mainbddass} implies the existence of a minimizer and, more generally, relative compactness of minimizing sequences.

\begin{theorem}\label{mainbdd}
Let $n\geq 2$ and $1/2\leq\sigma<1$. Let $\Omega\subset\R^n$ be a bounded open set with $C^1$ boundary and assume that
$$
0< S_{n,\sigma}(\Omega)< S_{n,\sigma}(\R^n_+) \,.
$$
Then any minimizing sequence for $S_{n,\sigma}(\Omega)$, normalized in $\mathring H^{\sigma}(\Omega)$, is relatively compact in $\mathring H^{\sigma}(\Omega)$. In particular, the infimum is attained.
\end{theorem}

We will also show that assumption \eqref{eq:mainbddass} is not only sufficient, but also necessary for the relative compactness of all minimizing sequences.

The assumption $S_{n,\sigma}(\Omega)>0$ is only needed for $\sigma=1/2$, since it holds automatically for $\sigma>1/2$, as recalled above. Moreover, we assume $n\geq 2$ and $1/2\leq\sigma<1$, because for $\sigma<1/2$ one has $S_{n,\sigma}(\Omega)=0$ by Lemma \ref{nosob}.

\medskip

Let us comment on the method of proof of our main results. Theorems \ref{mainhsbinding} and \ref{mainbddbinding} are proved by a trial function computation. We take the minimizers for the $S_{n,\sigma}(\R^n)$ and the $S_{n,\sigma}(\R^n_+)$ problem, respectively, scale them to a small ball and cut them off. To leading order, they will give the value of $S_{n,\sigma}(\R^n)$ and $S_{n,\sigma}(\R^n_+)$, and our goal is to compute the sub-leading correction. The computation is relatively straightforward in the proof of Theorem \ref{mainhsbinding} since the optimizer for the $S_{n,\sigma}(\R^n)$ is explicitly known. On the other hand, in the proof of Theorem \ref{mainbddbinding} we need to work with the unkown optimizer for $S_{n,\sigma}(\R^n_+)$ and it is crucial to have bounds on its behavior at infinity and near the boundary. These bounds are obtained by analyzing the Euler--Lagrange equation corresponding to the problem. Note that, since $I_{n,\sigma,\R^n_+}[u]\geq I_{n,\sigma,\R^n_+}[|u|]$ for all $u\in\mathring H^\sigma(\R^n_+)$, we may assume that the minimizer is non-negative. For non-negative functions, the Euler--Lagrange equation reads, after an appropriate normalization,
\begin{equation}
\label{eq:minimizer}
2 \int_{\R^n_+} \frac{u(x)-u(y)}{|x-y|^{n+2\sigma}}\,\ud y = u(x)^\frac{n+2\sigma}{n-2\sigma}
\qquad\text{for}\ x\in\R^n_+ \,.
\end{equation}
Here and in all the following the integral on the left side is understood in the principal value sense as the limit as $\epsilon\to 0$ of the integrals over $|x-y|>\epsilon$.

\begin{proposition}\label{prop:extremalfunct}
Let $n\geq 2$ and $1/2<\sigma<1$. Let $0\not\equiv u\in \mathring H^\sigma(\R^n_+)$ be non-negative and satisfy \eqref{eq:minimizer}. Then there are constants $0<c\leq C<\infty$ (depending on $u$) such that
\[
c \frac{x_n^{2\sigma-1}} {(1+|x|)^{n+2\sigma-2}} \leq u(x) \le C \frac{x_n^{2\sigma-1}} {(1+|x|)^{n+2\sigma-2}}\quad\mbox{for }x\in\R^n_+.
\]
\end{proposition}

It would be interesting to understand the behavior of non-negative solutions of \eqref{eq:minimizer} for $0<\sigma\leq 1/2$. The assumption $\sigma>1/2$ in Proposition \ref{prop:extremalfunct} leads to the same assumption in Theorem \ref{mainbddbinding}.

In order to prove Theorems \ref{mainhs} and \ref{mainbdd} we use the method of the missing mass, an optimization strategy that goes back to Lieb's work \cite{Lieb} as well as his contribution to \cite{BrNi}; see also \cite{BrLi}. Early uses of this method are in \cite{BrLi2,FrLiLo} and more recent ones, for instance, in \cite{TeTi,dVWi,FaVeVi,BeFrVi,FrLi4,FrLiSa}. The intuition, which is easier to explain in the context of Theorem \ref{mainbdd}, is that if a minimizing sequence goes weakly to zero in $\mathring H^\sigma(\Omega)$, then the sequence either concentrates in the interior of the domain or at the boundary and therefore the minimization problem looks `almost' like that on $\R^n$ or on $\R^n_+$. Thus, the strict inequality \eqref{eq:mainbddass} (together with the fact that $S_{n,\sigma}(\R^n_+)\leq S_{n,\sigma}(\R^n)$) excludes this behavior and therefore we have a non-zero weak limit. The non-linear structure of the minimization problem allows to upgrade this weak convergence to strong convergence, thereby proving Theorem \ref{mainbdd}.

The proof of Theorem \ref{mainhs} follows the same idea, but is technically more involved because of the invariance of the problem under dilations and translations parallel to the boundary. These symmetries allow a sequence of functions to go weakly to zero but the only interesting behavior is if a sequence goes to zero in a different way. This is formalized through the notion of weak convergence modulo symmetries. The intuition is that the only sequences that go to zero modulo symmetries are sequences that move away from the boundary in such a way that the problem looks `almost' like that on $\R^n$. Thus, the strict inequality \eqref{eq:bindingineq} excludes this behavior and we have a non-zero weak limit modulo symmetries. The rest of the proof is as in the case of a bounded domain. We note that the analysis here has similarities to that of the Hardy--Sobolev--Maz'ya inequality in \cite{TeTi} and of the Stein--Tomas inequality \cite{FrLiSa}, where one also has to consider weak convergence modulo the symmetries of the problem.


\subsection*{Acknowledgements}

Part of this work was done when T. J. was visiting California Institute of Technology as an Orr foundation Caltech-HKUST Visiting Scholar during 2015-2016. He would like to thank Professor Thomas Y. Hou for hosting his visit. He also thanks Professors Zhen-Qing Chen and  Dong Li for useful discussions. Partial support through National Science Foundation, grant DMS-1363432 (R.L.F.), Hong Kong RGC grant ECS 26300716 (T.J.) and NSFC 11501034, a key project of NSFC 11631002 and NSFC 11571019, (J.X.) is acknowledged.


\section{Verifying the strict inequality \eqref{eq:bindingineq}}

Our goal in this section is to prove Theorem \ref{mainhsbinding}. As a warm-up we prove a much simpler result, namely that the \emph{non-strict} inequality $S_{n,\sigma}(\Omega)\leq S_{n,\sigma}(\R^n)$ holds on any open set $\Omega$ without any additional assumptions on $n$ and $\sigma$.

\begin{lemma}\label{bindingsimple}
Let $0<\sigma<1/2$ if $n=1$ and $0<\sigma<1$ if $n\geq 2$ and let $\Omega\subset\R^n$ be an open set. Then
$$
S_{n,\sigma}(\Omega) \leq S_{n,\sigma}(\R^n) \,.
$$
\end{lemma}

\begin{proof}
After a translation we may assume that $0\in\Omega$. Let $0\not\equiv W\in C^1_c(\R^n)$ with compact support and set $W_\lambda(x):= \lambda^\frac{n-2\sigma}{2} W(\lambda x)$. Then $W_\lambda\in C^1_c(\Omega)$ for all sufficiently large $\lambda$ and we have
$$
I_{n,\sigma,\Omega}[W_\lambda] \leq I_{n,\sigma,\R^n}[W_\lambda] = I_{n,\sigma,\R^n}[W] \,.
$$
Moreover,
$$
\int_\Omega W_\lambda(x)^\frac{2n}{n-2\sigma}\,\ud x = \int_{\R^n} W(x)^\frac{2n}{n-2\sigma}\,\ud x \,, 
$$
and therefore
$$
S_{n,\sigma}(\Omega) \leq \frac{I_{n,\sigma,\R^n}[W]}{\left( \int_{\R^n} W(x)^\frac{2n}{n-2\sigma}\,\ud x \right)^\frac{n-2\sigma}{n}} \,.
$$
We now take the infimum over all $W\in C^1_c(\R^n)$ and using the density of $C^1_c(\R^n)$ in $\mathring H^\sigma(\R^n)$, we obtain the claim.
\end{proof}

The following proposition implies, in particular, Theorem \ref{mainhsbinding}.

\begin{proposition}\label{stricths}
Let $0<\sigma<1/2$ if $n=1$ and $0<\sigma<1$ if $n\geq 2$ and assume that
$$
n\geq 4\sigma \,.
$$
Let $\Omega\subset\R^n$ be an open set such that $\Omega^c$ has non-empty interior. Then
$$
S_{n,\sigma}(\Omega) < S_{n,\sigma}(\R^n) \,.
$$
\end{proposition}

For the proof we will need the fact, recalled in the introduction, that the optimal constant $S_{n,\sigma}(\R^n)$ is achieved by multiples, translates and dilates of the function $(1+|x|^2)^{-\frac{n-2\sigma}{2}}$. For $\lambda>0$ let
\[
U_\lda(x)=c_0\left(\frac{\lda}{1+\lda^2|x|^2}\right)^\frac{n-2\sigma}{2},
\]
where $c_0$ is a constant (independent of $\lda$) such that $\|U_\lda\|_{L^\frac{2n}{n-2\sigma}(\R^n)}=1$. The Euler--Lagrange equation of the minimization problem reads
\be\label{eq:ELa}
2 \int_{\R^n}\frac{U_\lda(x)-U_\lda(y)}{|x-y|^{n+2\sigma}}\,\ud y= S_{n,\sigma}(\R^n)\ U_\lda(x)^{\frac{n+2\sigma}{n-2\sigma}} \,.
\ee
We recall that the integral on the left side is understood in a principal value sense.

\begin{proof}
We denote by $B_r$ the open ball in $\R^n$ centered at the origin with radius $r>0$. Since the assumption and the conclusion of the proposition are invariant with respect to translations and dilations of $\Omega$, we may assume that $B_4\subset \om$. 

Let $\eta$ be a radial $C^1$ function such that $\eta\equiv 1$ in $B_2$, $0\leq\eta\leq 1$ in $B_3$ and $\eta\equiv 0$ in $B_3^c$, and set
$$
u_\lda=\eta U_\lda \,.
$$
This function belongs to $C^1_c(\Omega)$ and we will estimate $\| u_\lambda\|_\frac{2n}{n-2\sigma}$ and $I_{n,\sigma,\Omega}[u_\lambda]$ as $\lambda\to\infty$ in order to get an upper bound for $S_{n,\sigma}(\Omega)$.

By the normalization and the decay of $U_\lambda$, it is easy to see that
\be\label{eq:Error1}
\int_{\om} u_\lda^{\frac{2n}{n-2\sigma}}\,\ud x= 1- O(\lda^{-n}) \,.
\ee

In order to bound $I_{n,\sigma,\Omega}[u_\lambda]$ we write
\be
\label{eq:Error2a}
I_{n,\sigma,\Omega}[u_\lambda] = \int_{\om } u_\lda(x)f_\lambda(x)\,\ud x
\ee
with
$$
f_\lambda(x) := 2\int_{\om} \frac{u_\lda(x)-u_\lda(y)}{|x-y|^{n+2\sigma}}\,\ud y
$$
and estimate $f_\lambda$ pointwise in the regions $B_1$ and $B_3\setminus B_1$.

For $x\in B_1$ we have
\be \label{eq:Error2}
\begin{split}
f_\lambda(x) & = 2 \int_{\R^n} \frac{U_\lda(x)-U_\lda(y)}{|x-y|^{n+2\sigma}}\,\ud y
+  2 \int_{\R^n} \frac{U_\lda(y)-u_\lda (y)}{|x-y|^{n+2\sigma}}\,\ud y
-2 \int_{\R^n\setminus \om} \frac{U_\lda(x)}{|x-y|^{n+2\sigma}}\,\ud y\\
& =S_{n,\sigma}(\R^n) U_{\lda}(x)^{\frac{n+2\sigma}{n-2\sigma}} 
+ 2\int_{\R^n\setminus B_2} \frac{U_\lda(y)-u_\lda (y)}{|x-y|^{n+2\sigma}}\,\ud y + V_\Omega(x)U_\lda(x)\\
& = S_{n,\sigma}(\R^n) U_{\lda}(x)^{\frac{n+2\sigma}{n-2\sigma}} +O(\lda^{-\frac{n-2\sigma}{2}}) + V_\Omega(x) U_\lda(x) \,.
\end{split}
\ee
with
\begin{equation}
\label{eq:potential}
V_\Omega(x) = - 2 \int_{\Omega^c} \frac{\ud y}{|x-y|^{n+2\sigma}} \,.
\end{equation}
Since $\Omega^c$ has non-empty interior, we have $V_\Omega(x) \leq -\epsilon_0<0$ for all $x\in B_1$.

For $x\in B_3 \setminus B_{1}$, we have
\begin{equation*}
\begin{split}
f_\lambda(x) &
= 2 \int_{\om} \frac{U_\lda(x)-U_\lda(y)}{|x-y|^{n+2\sigma}}\,\ud y
+ 2 \int_{\om} \frac{(u_\lda-U_\lda)(x)-(u_\lda-U_\lda)(y)}{|x-y|^{n+2\sigma}}\,\ud y \\
& = 2 \int_{\R^n} \frac{U_\lda(x)-U_\lda(y)}{|x-y|^{n+2\sigma}}\,\ud y - 2 \int_{\R^n\setminus \om} \frac{U_\lda(x)-U_\lda (y)}{|x-y|^{n+2\sigma}}\,\ud y\\&
\quad+2\int_{\om} \frac{(u_\lda-U_\lda)(x)-(u_\lda-U_\lda)(y)}{|x-y|^{n+2\sigma}}\,\ud y \,.
\end{split}
\end{equation*}
Since $x\in B_3$ and $B_4\subset\Omega$, we have
\[
\begin{split}
- \int_{\R^n\setminus \om} \frac{U_\lda(x)-U_\lda (y)}{|x-y|^{n+2\sigma}}\,\ud y \leq \int_{\R^n\setminus \om} \frac{U_\lda (y)}{|x-y|^{n+2\sigma}}\,\ud y \le C\int_{\R^n\setminus \om} \frac{U_\lda (y)}{|y|^{n+2\sigma}}\,\ud y=O(\lda^{-\frac{n-2\sigma}{2}}) \,.
\end{split}
\]
Also,
\[
\begin{split}
\int_{\om} \frac{(u_\lda-U_\lda)(x)-(u_\lda-U_\lda)(y)}{|x-y|^{n+2\sigma}}\,\ud y 
&\le \int_{B_1(x)} \frac{((\eta-1)U_\lda)(x)-((\eta-1)U_\lda)(y)}{|x-y|^{n+2\sigma}}\,\ud y \\
&\quad+\int_{\om\setminus B_1(x)} \frac{((1-\eta)U_\lda)(y)}{|x-y|^{n+2\sigma}}\,\ud y\\
&\le C \max_{B_1(x)}|\nabla ((\eta-1)U_\lda)|+O(\lda^{-\frac{n-2\sigma}{2}})= O(\lda^{-\frac{n-2\sigma}{2}}) \,.
\end{split}
\]
To summarize, for $x\in B_3 \setminus B_{1}$, we have
\begin{equation}
\label{eq:Error3}
f_\lambda(x) \leq S_{n,\sigma}(\R^n) U_{\lda}(x)^{\frac{n+2\sigma}{n-2\sigma}} +O(\lda^{-\frac{n-2\sigma}{2}}) \,.
\end{equation}

Inserting the pointwise bounds \eqref{eq:Error2} and \eqref{eq:Error3} into \eqref{eq:Error2a} we obtain that
\be
\label{eq:Error4}
\begin{split}
I_{n,\sigma,\Omega}[u_\lambda] & \leq \int_{B_{1}}U_{\lda}(x) \left( S_{n,\sigma}(\R^n) U_{\lda}(x)^{\frac{n+2\sigma}{n-2\sigma}} + O(\lda^{-\frac{n-2\sigma}{2}})- \epsilon_0 U_\lda(x) \right) \ud x\\&
\quad + \int_{B_3\setminus B_{1}} u_{\lda}(x) \left( S_{n,\sigma}(\R^n) U_{\lda}(x)^{\frac{n+2\sigma}{n-2\sigma}} +O(\lda^{-\frac{n-2\sigma}{2}}) \right) \ud x\\
& \le S_{n,\sigma}(\R^n_+) \int_{B_3} U_\lda^{\frac{2n}{n-2\sigma}}\,\ud x - \va_0 \int_{B_{1}} U_\lda^{2}\,\ud x+ C \lda^{-\frac{n-2\sigma}{2}}\int_{B_{3}} U_{\lda}\,\ud x \\
& \le
\begin{cases}
S_{n,\sigma}(\R^n) -\frac{\va_0}{C} \lda^{-2\sigma}+ C\lda^{-n+2\sigma} \,,\qquad\mbox{if }n>4\sigma,\\
S_{n,\sigma}(\R^n) -\frac{\va_0}{C} \lda^{-2\sigma}\log\lambda+ C\lda^{-n+2\sigma}\,,\qquad\mbox{if }n=4\sigma \,.
\end{cases}
\end{split}
\ee
In the last bound we used $\int_{B_3} U_\lda^{\frac{2n}{n-2\sigma}}\,\ud x\leq 1$, which is analogous to \eqref{eq:Error1}, as well as 
$$
\int_{B_{1}} U_\lda^2\,\ud x\geq 
\begin{cases}
C^{-1} \lda^{-2\sigma} \,,\qquad\mbox{if }n>4\sigma,\\
C^{-1} \lda^{-2\sigma}\log\lambda\,,\qquad\mbox{if }n=4\sigma \,.
\end{cases}
$$
Combining \eqref{eq:Error1} and \eqref{eq:Error4}, we find
\begin{align*}
S_{n,\sigma}(\Omega)&\le \frac{I_{n,\sigma,\Omega}[u_\lambda]}{\left(\int_{\om} u_\lda^{\frac{2n}{n-2\sigma}}\,\ud x\right)^{\frac{n-2\sigma}{n}}} \\
& \le (1-C\lambda^{-n}) \times
\begin{cases}
S_{n,\sigma}(\R^n) -\frac{\va_0}{C} \lda^{-2\sigma}+ C\lda^{-n+2\sigma} \quad\mbox{if }n>4\sigma \,,\\
S_{n,\sigma}(\R^n) -\frac{\va_0}{C} \lda^{-2\sigma}\log\lambda+ C\lda^{-n+2\sigma} \quad\mbox{if }n=4\sigma \,.\\
\end{cases}
\end{align*}
The right side is strictly less than $S_{n,\sigma}(\R^n)$ provided that $\lda$ is sufficiently large. This completes the proof of the proposition.
\end{proof}


\section{Existence of a minimizer}

Our goal in this section is to prove Theorem \ref{mainhs}. Let $(u_k)\subset\mathring H^\sigma(\R^n_+)$. We define
$$
u_k \rightharpoonup_{\text{symm}} 0
\qquad\text{in}\ \mathring H^\sigma(\R^n_+)
$$
if for any sequences $(\lambda_k)\subset (0,\infty)$, $(a_k)\subset\R^{n-1}$ one has
$$
\lambda_k^\frac{n-2\sigma}{2} u_k(\lambda_k(x'-a_k),\lambda_k x_n) \rightharpoonup 0
\qquad\text{in}\ \mathring H^\sigma(\R^n_+) \,.
$$
Moreover, we define
$$
S^*_{n,\sigma}(\R^n_+) := \inf \left\{ \liminf_{k\to\infty} \left( \int_{\R^n_+} |u_k|^\frac{2n}{n-2\sigma} \,\ud x \right)^{-\frac{n-2\sigma}{n}} :\ I_{n,\sigma,\R^n_+}[u_k] =1 \,,\ u_k \rightharpoonup_{\text{symm}} 0 \right\}.
$$
We shall see shortly that there are, indeed, sequences $(u_k)$ with $I_{n,\sigma,\R^n_+}[u_k] =1$ and $u_k\rightharpoonup_{\text{symm}} 0$, so the infimum is well-defined.

The key step in the proof of Theorem \ref{mainhs} is

\begin{proposition}\label{mainhsprop}
Let $0<\sigma<1/2$ if $n=1$ and $\sigma\in(0,1/2)\cup(1/2,1)$ if $n\geq 2$. Then
$$
S^*_{n,\sigma}(\R^n_+) = S_{n,\sigma}(\R^n) \,.
$$
\end{proposition}

The assumption $\sigma\neq1/2$ if $n\geq 2$ comes from the use of Hardy's inequality, both in Steps 1 and 3 of the proof.

Given this proposition it is easy to conclude the

\begin{proof}[Proof of Theorem \ref{mainhs}]
Let $(u_k)\subset\mathring H^\sigma(\R^n_+)$ be a minimizing sequence for $S_{n,\sigma}(\R^n_+)$ with $I_{n,\sigma,\R^n_+}[u_k] =1$ for all $k$. Assumption \eqref{eq:bindingineq} implies that $u_k\not\rightharpoonup_{\text{symm}} 0$, that is, after passing to a subsequence there are $(\lambda_k)\subset(0,\infty)$, $(a_k)\subset\R^{n-1}$ and $0\not\equiv v\in\mathring H^\sigma(\R^n_+)$ such that
$$
v_k(x) := \lambda_k^\frac{n-2\sigma}{2} u_k(\lambda_k(x'-a_k),\lambda_k x_n) \rightharpoonup v
\qquad\text{in}\ \mathring H^\sigma(\R^n_+) \,.
$$
Moreover, by Rellich's theorem after passing to a subsequence if necessary, $v_k\to v$ almost everywhere. Let $r_k:=v-v_k$. Then, by weak convergence in $\mathring H^\sigma(\R^n_+)$,
$$
1 = I_{n,\sigma,\R^n_+}[u_k] =I_{n,\sigma,\R^n_+}[v_k] = I_{n,\sigma,\R^n_+}[v] + I_{n,\sigma,\R^n_+}[r_k] + o(1) \,.
$$
Thus, $I_{n,\sigma,\R^n_+}[r_k]$ converges and
\begin{equation}
\label{eq:mmm1a}
T := \lim_{k\to\infty} I_{n,\sigma,\R^n_+}[r_k]
\qquad\text{satisfies}\qquad
1= I_{n,\sigma,\R^n_+}[v] + T \,.
\end{equation}
Moreover, by almost everywhere convergence and the Br\'ezis--Lieb lemma \cite{BrLi},
\begin{align*}
S_{n,\sigma}(\R^n_+)^{-\frac{n-2\sigma}{n}} + o(1) & = \int_{\R^n_+} |u_k|^\frac{2n}{n-2\sigma}\,\ud x \\
& = \int_{\R^n_+} |v_k|^\frac{2n}{n-2\sigma}\,\ud x \\
& = \int_{\R^n_+} |v|^\frac{2n}{n-2\sigma}\,\ud x + \int_{\R^n_+} |r_k|^\frac{2n}{n-2\sigma}\,\ud x + o(1) \,. 
\end{align*}
Thus, $\int_{\R^n_+} |r_k|^\frac{2n}{n-2\sigma}\,\ud x$ converges and
\begin{equation}
\label{eq:mmm2a}
M := \lim_{k\to\infty} \int_{\R^n_+} |r_k|^\frac{2n}{n-2\sigma}\,\ud x 
\qquad\text{satisfies}\qquad
S_{n,\sigma}(\R^n_+)^{-\frac{n-2\sigma}{n}}  = \int_{\R^n_+} |v|^\frac{2n}{n-2\sigma}\,\ud x + M \,.
\end{equation}
Clearly, by Sobolev's inequality we have
\begin{equation}
\label{eq:mmm3a}
T \geq S_{n,\sigma}(\R^n_+)\ M^\frac{n-2\sigma}{n} \,.
\end{equation}

Given \eqref{eq:mmm1a}, \eqref{eq:mmm2a} and \eqref{eq:mmm3a} the proof is concluded by a standard argument. We use the elementary fact that for $0\leq\theta\leq 1$,
\begin{equation}
\label{eq:elem}
(a-b)^\theta \geq a^\theta - b^\theta
\qquad\text{for all}\ a\geq b\geq 0 \,.
\end{equation}
Applying this with $\theta=(n-2\sigma)/n$ we find
\begin{align*}
1 & = I_{n,\sigma,\R^n_+}[v] + T \\
& \geq I_{n,\sigma,\R^n_+}[v] + S_{n,\sigma}(\R^n_+) M^\frac{n-2\sigma}{n} \\
& = I_{n,\sigma,\R^n_+}[v] + S_{n,\sigma}(\R^n_+) \left(S_{n,\sigma}(\R^n_+)^{-\frac{n-2\sigma}{n}} - \int_{\R^n_+} |v|^\frac{2n}{n-2\sigma}\,\ud x \right)^\frac{n-2\sigma}{n} \\
& \geq I_{n,\sigma,\R^n_+}[v] + 1 - S_{n,\sigma}(\R^n_+) \left( \int_{\R^n_+} |v|^\frac{2n}{n-2\sigma}\,\ud x \right)^\frac{n-2\sigma}{n} \,.
\end{align*}
Thus, we have shown that
$$
I_{n,\sigma,\R^n_+}[v] - S_{n,\sigma}(\R^n_+) \left( \int_{\R^n_+} |v|^\frac{2n}{n-2\sigma}\,\ud x \right)^\frac{n-2\sigma}{n} \leq 0 \,.
$$
Thus, equality must hold everywhere and $v$ is an optimizer. Since equality in \eqref{eq:elem} holds only if $b=0$ or $a=b$, we conclude that $M=0$, and then equality in \eqref{eq:mmm3a} implies that $T=0$. This means that $I_{n,\sigma,\R^n_+}[v]=1$ and therefore $(v_k)$ converges, in fact, \emph{strongly} in $\mathring H^\sigma(\Omega)$ to $v$. This completes the proof of the theorem.
\end{proof}

Thus we are left with the proof of Proposition \ref{mainhsprop}. We begin the proof with an auxiliary result that yields a typical sequence of functions that tends to zero in the sense of $\rightharpoonup_{\mathrm{symm}}$.

\begin{lemma}\label{convmodsymm}
Let $(c_k)\subset\R$ be a sequence with $c_k\to\infty$ and let $(w_k)\subset\mathring H^\sigma(\R^n)$ be a sequence which converges in $\mathring H^\sigma(\R^n)$ and which satisfies $\supp w_k\subset\{ x_n\geq -c_k\}$ for all $k$. Then
$$
w_k(x',x_n-c_k) \rightharpoonup_{\mathrm{symm}} 0
\qquad\text{in}\ \mathring H^\sigma(\R^n_+) \,.
$$
\end{lemma}

\begin{proof}
We abbreviate $\tilde w_k(x)= w_k(x',x_n-c_k)$. We shall show that for any sequence $(\lambda_n)\subset(0,\infty)$, $(a_k)\subset\R^{n-1}$ one has
$$
\lambda_k^\frac{n-2\sigma}{2} \tilde w_k(\lambda_k(x'-a_k),\lambda_k x_n) \rightharpoonup 0 \qquad\text{in}\ L^\frac{2n}{n-2\sigma}(\R^n_+) \,.
$$
This implies the result, for if $v$ denotes any weak limit point in $\mathring H^\sigma(\R^n_+)$ of $\lambda_k^\frac{n-2\sigma}{2} \tilde w_k(\lambda_k(x'-\alpha_k),\lambda_k x_n)$ (which exists by weak compactness), then by Sobolev's theorem $v$ is also a weak limit point in $L^\frac{2n}{n-2\sigma}(\R^n_+)$ and therefore, by what we shall show, $v=0$.

Let $f\in L^\frac{2n}{n+2\sigma}(\R^n_+)$ (extended by zero to $\R^n$) and
$$
f_k(y',y_n) = \lambda_k^{-\frac{n-2\sigma}{2}} f(a_k + \lambda_k^{-1}y', \lambda_k^{-1}(y_n+c_k)) \,,
$$
so that
$$
\int_{\R^n_+} f(x) \lambda_k^\frac{n-2\sigma}{2} \tilde w_k(\lambda_k(x'-a_k),\lambda_k x_n) \,\ud x = \int_{\R^n} f_k(y) w_k(y)\,\ud y \,.
$$
Since $c_k\to\infty$, we have $f_k\rightharpoonup 0$ in $L^\frac{2n}{n+2\sigma}(\R^n)$. On the other hand, by Sobolev's inequality $w_n$ converges strongly in $L^\frac{2n}{n-2\sigma}(\R^n)$, and therefore $\int_{\R^n} f_k(y) w_k(y)\,\ud y \to 0$, which proves the claim.
\end{proof}


\subsection{Proof of Proposition \ref{mainhsprop}. Part 1}

We begin with the proof of the inequality
$$
S^*_{n,\sigma}(\R^n_+) \geq S_{n,\sigma}(\R^n) \,.
$$

\emph{Step 1.}
By a diagonal argument we can find a sequence $(u_k)\subset\mathring H^\sigma(\R^n_+)$ with $I_{n,\sigma,\R^n_+}[u_k]=1$ for all $k$, $u_k \rightharpoonup_{\text{symm}} 0$ and
\begin{equation}
\label{eq:sstarseq}
\lim_{k\to\infty} \int_{\R^n} |u_k|^\frac{2n}{n-2\sigma}\,\ud x \to \left( S^*_{n,\sigma}(\R^n_+) \right)^{-\frac{n}{n-2\sigma}}.
\end{equation}
We extend $u_k$ by zero to the lower half-plane and note that
$$
I_{n,\sigma,\R^n}[u_k] = I_{n,\sigma,\R^n_+}[u_k] + \kappa_{n,\sigma} \int_{\R^n_+} \frac{u_k^2}{x_n^{2\sigma}} \,\ud x \,.
$$
Therefore the normalization $I_{n,\sigma,\R^n_+}[u_k]=1$ and Hardy's inequality \cite{BoDy} imply that
\begin{equation}
\label{eq:bdd}
\sup_k I_{n,\sigma,\R^n}[u_k] <\infty \,.
\end{equation}
On the other hand, from \eqref{eq:sstarseq} and the improved Sobolev inequality of Lemma \ref{lem:improvedinequality} we conclude that 
$$
\liminf_{k\to\infty} \sup_{t>0} t^\frac{n-2\sigma}{4} \| e^{t\Delta} u_k \|_{L^\infty(\R^n)} > 0 \,.
$$
Thus, there are $(t_k)\subset(0,\infty)$, $(a_k)\subset\R^{n-1}$ and $(b_k)\subset\R$ such that
\begin{equation}
\label{eq:besovnonzero}
\liminf_{k\to\infty} t_k^\frac{n-2\sigma}{4} \left| e^{t_k\Delta} u_k(a_k,b_k) \right| > 0 \,.
\end{equation}
Let
$$
v_k(x',x_n) := t_k^\frac{n-2\sigma}{4} u_k(a_k + t_k^\frac12 x',b_k +t_k^\frac12x_n ) \,.
$$
Because of $I_{n,\sigma,\R^n}[v_k] = I_{n,\sigma,\R^n}[u_k]$, the boundedness \eqref{eq:bdd} and weak compactness, after passing to a subsequence if necessary, we may assume that $v_k \rightharpoonup v$ in $\mathring H^\sigma(\R^n)$. By Rellich's theorem, after possibly passing to another subsequence, we may also assume that $v_k \to v$ almost everywhere. By Sobolev's theorem, weak convergence in $\mathring H^\sigma(\R^n)$ implies weak convergence in $L^\frac{2n}{n-2\sigma}(\R^n)$, and therefore the identity
$$
t_k^\frac{n-2\sigma}{4} \left( e^{t_k\Delta} u_k \right)(a_k,b_k) = (4\pi)^{-n/2} \int_{\R^n} e^{-|x|^2/4} v_k(x)\,\ud x
$$
together with the fact that $e^{-|x|^2/4}\in L^\frac{2n}{n+2\sigma}(\R^n)$ yields
$$
\lim_{k\to\infty} t_k^\frac{n-2\sigma}{4} \left( e^{t_k\Delta} u_k\right)(a_k,b_k) = (4\pi)^{-n/2} \int_{\R^n} e^{-|x|^2/4} v(x)\,\ud x \,.
$$
The bound \eqref{eq:besovnonzero} now implies that $v\not\equiv 0$.

\emph{Step 2.}
We abbreviate $c_k := t_k^{-\frac12}b_k$ and claim that $c_k\to\infty$. We prove this by contradiction. Indeed, if we had $\limsup_{k\to\infty} |c_k|<\infty$, then along a subsequence $c_k\to c$ and $t_k^\frac{n-2\sigma}{4} u_k(a_k+t_k^\frac12 x',t_k^\frac12x_n)\rightharpoonup v(x',x_n-c)$ in $\mathring H^\sigma(\R^n)$ and therefore in $\mathring H^\sigma(\R^n_+)$. Since $u_k\rightharpoonup_{\text{symm}} 0$, we conclude that $v(x',x_n-c)\equiv 0$ on $\R^n_+$. Since $t_k^\frac{n-2\sigma}{4} u_k(a_k+t_k^\frac12 x',t_k^\frac12x_n)\equiv 0$ on $\R^n_-$, we also have $v(x',x_n-c)\equiv 0$ on $\R^n_-$ and therefore $v\equiv 0$, a contradiction. If we had $\liminf_{k\to\infty} c_k = -\infty$, then along a subsequence we had for every $\phi\in L^\frac{2n}{n+2\sigma}(\R^n)$ 
$$
\int_{\R^n} \phi v \,\ud x = \lim_{k\to\infty} \int_{\R^n} \phi v_k\,\ud x = \lim_{k\to\infty} \int_{x_n \geq - c_k} \phi v_k\,\ud x = 0 \,,
$$
since $\chi_{\{x_n \geq - c_k\}}\phi\to 0$ strongly in $L^\frac{2n}{n+2\sigma}(\R^n)$ by dominated convergence. Thus, again $v\equiv 0$, a contradiction. We therefore have shown that $c_k\to\infty$.

\emph{Step 3.}
Since compactly supported functions are dense in $\mathring H^\sigma(\R^n)$, there is a sequence $(w_k)\subset\mathring H^\sigma(\R^n)$ with $w_k\to v$ in $\mathring H^\sigma(\R^n)$ and $\supp w_k \subset\{ |x'|\leq c_k \,,\ |x_n|\leq c_k/2\}$. Either by construction or by Rellich's theorem after passing to a subsequence, we may assume that $w_k\to v$ almost everywhere.

Let us introduce the translated functions
$$
\tilde v_k(x',x_n) := v_k(x',x_n-c_k) \,,
\qquad
\tilde w_k(x',x_n) := w_k(x',x_n-c_k)
$$
and set $\tilde r_k := \tilde v_k - \tilde w_k$. Note that $\tilde w_k\in\mathring H^\sigma(\R^n_+)$ and therefore also $\tilde r_k\in\mathring H^\sigma(\R^n_+)$.

We claim that $I_{n,\sigma,\R^n_+}[\tilde r_k]$ converges and that
\begin{equation}
\label{eq:mmm1b}
T := \lim_{k\to\infty}I_{n,\sigma,\R^n}[\tilde r_k]
\qquad\text{satisfies}\qquad
1 = I_{n,\sigma,\R^n}[v] + T \,.
\end{equation}
To see this, we write
$$
1= I_{n,\sigma,\R^n_+}[u_k] = I_{n,\sigma,\R^n_+}[\tilde v_k] =I_{n,\sigma,\R^n_+}[\tilde w_k] + I_{n,\sigma,\R^n_+}[\tilde r_k] + \mathcal R_k
$$
with the remainder
$$
\mathcal R_k := 2 I_{n,\sigma,\R^n_+}[\tilde r_k,\tilde w_k] \,.
$$
Here we have introduced the natural bilinear form associated to the quadratic form $I_{n,\sigma,\R^n_+}$. Clearly, $w_k\to v$ in $\mathring H^\sigma(\R^n)$ implies that
$$
I_{n,\sigma,\R^n_+}[\tilde w_k] = \iint_{x_n>-c_k\,,\ y_n>-c_k} \frac{(w_k(x)-w_k(y))^2}{|x-y|^{n+2\sigma}}\,\ud x\,\ud y \to I_{n,\sigma,\R^n}[v] \,,
$$
and it remains to prove $\mathcal R_k\to 0$. We write
$$
\mathcal R_k = 2 I_{n,\sigma,\R^n}[\tilde r_k,\tilde w_k] - 2 \kappa_{n,\sigma} \int_{x_n>0} \frac{\tilde r_k \tilde w_k}{x_n^{2\sigma}} \,\ud x \,.
$$
Since $v_k\rightharpoonup v$ in $\mathring H^\sigma(\R^n)$ and $w_k\to v$ in $\mathring H^\sigma(\R^n)$, we have
$$
I_{n,\sigma,\R^n}[\tilde r_k,\tilde w_k] = I_{n,\sigma,\R^n}[v_k-w_k,w_k]
= I_{n,\sigma,\R^n}[v_k-v,w_k] + I_{n,\sigma,\R^n}[v-w_k,w_k] \to 0 \,.
$$
Moreover, by Hardy's inequality \cite{BoDy} we have
\begin{align*}
\left| \int_{x_n>0} \frac{\tilde r_k \tilde w_k}{x_n^{2\sigma}} \,\ud x \right| & \leq \left( \int_{x_n>0} \frac{\tilde r_k^2}{x_n^{2\sigma}} \,\ud x \right)^{1/2} \left( \int_{x_n>0} \frac{\tilde w_k^2}{x_n^{2\sigma}} \,\ud x \right)^{1/2} \\
& \lesssim \left( I_{n,\sigma,\R^n_+}[\tilde r_k] \right)^{1/2} \left( \int_{x_n>0} \frac{\tilde w_k^2}{x_n^{2\sigma}} \,\ud x \right)^{1/2}
\end{align*}
and, since the first square root factor remains bounded as $k\to\infty$, it suffices to show that the second one tends to zero.

Let $0<\epsilon\leq 1/2$ and split
\begin{align*}
\int_{x_n>0} \frac{\tilde w_k^2}{x_n^{2\sigma}} \,\ud x
& = \int_{x_n>-c_k} \frac{w_k^2}{(x_n+c_k)^{2\sigma}} \,\ud x \\
& \leq \int_{|x_n|< \epsilon c_k} \frac{w_k^2}{(x_n+c_k)^{2\sigma}} \,\ud x
+ \int_{|x_n|\geq \epsilon c_k} \frac{w_k^2}{(x_n+c_k)^{2\sigma}} \,\ud x \,.
\end{align*}
By the support properties of $w_k$ we can bound
$$
\int_{|x_n|\geq \epsilon c_k} \frac{w_k^2}{(x_n+c_k)^{2\sigma}} \,\ud x
\leq \left( \int_{|x_n|\geq \epsilon c_k} |w_k|^\frac{2n}{n-2\sigma}\,\ud x \right)^{\frac{n-2\sigma}{n}} \left( \int_{|x'|<c_k \,,\ |x_n|\leq c_k/2} \frac{\ud x}{(x_n+c_k)^n} \right)^\frac{2\sigma}{n}
$$
and
$$
\int_{|x'|<c_k \,,\ |x_n|\leq c_k/2} \frac{\ud x}{(x_n+c_k)^n} \lesssim
c_k^{n-1} \int_{|x_n|\leq c_k/2} \frac{\ud x_n}{(x_n+c_k)^n} = \int_{|t|\leq 1/2} \frac{\ud t}{(t+1)^n} \,.
$$
Thus, since by Sobolev $w_k\to w$ in $L^\frac{2n}{n-2\sigma}(\R^n)$, and consequently $\int_{|x_n|\geq \epsilon c_k} |w_k|^\frac{2n}{n-2\sigma}\,\ud x\to 0$, we find
$$
\lim_{k\to\infty} \int_{|x_n|\geq \epsilon c_k} \frac{w_k^2}{(x_n+c_k)^{2\sigma}} \,\ud x = 0 \,.
$$
On the other hand,
$$
\int_{|x_n|< \epsilon c_k} \frac{w_k^2}{(x_n+c_k)^{2\sigma}} \,\ud x
\leq \left( \int_{\R^n} |w_k|^\frac{2n}{n-2\sigma}\,\ud x \right)^{\frac{n-2\sigma}{n}} \left( \int_{|x'|<c_k \,,\ |x_n|< \epsilon c_k} \frac{\ud x}{(x_n+c_k)^n} \right)^\frac{2\sigma}{n}
$$
and
$$
\int_{|x'|<c_k \,,\ |x_n|\leq \epsilon c_k} \frac{\ud x}{(x_n+c_k)^n} \lesssim c_k^{n-1} \int_{|x_n|\leq \epsilon c_k} \frac{\ud x_n}{(x_n+c_k)^n}
= \int_{|t|\leq \epsilon} \frac{\ud t}{(t+1)^n} \lesssim \epsilon \,.
$$
This shows that
$$
\limsup_{k\to\infty} \int_{|x_n|< \epsilon c_k} \frac{w_k^2}{(x_n+c_k)^{2\sigma}} \,\ud x \lesssim \epsilon^\frac{2\sigma}{n} \,.
$$
Since $\epsilon>0$ can be chosen arbitrarily small, we conclude that $\mathcal R_k\to 0$.

\emph{Step 4.}
We claim that $\int_{\R^n_+} |\tilde r_k|^\frac{2n}{n-2\sigma}\,\ud x$ converges and that
\begin{equation}
\label{eq:mmm2b}
M:=\lim_{k\to\infty} \int_{\R^n_+} |\tilde r_k|^\frac{2n}{n-2\sigma}\,\ud x
\qquad\text{satisfies}\qquad
\left( S^*_{n,\sigma}(\R^n_+) \right)^{-\frac{n}{n-2\sigma}} = \int_{\R^n} |v|^\frac{2n}{n-2\sigma}\,\ud x + M \,.
\end{equation}
We recall that both $v_k$ and $w_k$ converge almost everywhere to $v$ and therefore $v_k-w_k$ converges almost everywhere to zero. Moreover, $w_k-v$ tends to zero in $L^\frac{2n}{n-2\sigma}(\R^n)$. Therefore by a slight generalization of the Br\'ezis--Lieb lemma \cite{BrLi}, which allows for an additional term that vanishes in the corresponding Lebesgue space \cite{FrLiSa}, we infer that
\begin{align*}
\left( S^*_{n,\sigma}(\R^n_+) \right)^{-\frac{n}{n-2\sigma}} + o(1) & = \int_{\R^n_+} |u_k|^\frac{2n}{n-2\sigma}\,\ud x = \int_{\R^n} |v_k|^\frac{2n}{n-2\sigma}\,\ud x \\
& = \int_{\R^n} |v|^\frac{2n}{n-2\sigma}\,\ud x + \int_{\R^n} |w_k-v_k|^\frac{2n}{n-2\sigma}\,\ud x + o(1) \,.
\end{align*}
Since
$$
\int_{\R^n} |w_k-v_k|^\frac{2n}{n-2\sigma}\,\ud x = \int_{\R^n} |\tilde r_k|^\frac{2n}{n-2\sigma}\,\ud x \,,
$$
we obtain the claim \eqref{eq:mmm2b}.

\emph{Step 5.} We claim that 
\begin{equation}
\label{eq:mmm3b}
T \geq S^*_{n,\sigma}(\R^n_+) M^\frac{n-2\sigma}{n} \,.
\end{equation}
If $M=0$, there is nothing to prove, so we may assume that $M>0$. We know from Lemma \ref{convmodsymm} that $\tilde w_k\rightharpoonup_{\text{symm}} 0$ in $\mathring H^\sigma(\R^n_+)$. Since $\tilde v_k\rightharpoonup_{\text{symm}} 0$ (which follows from $u_k\rightharpoonup_{\text{symm}} 0$ and the fact that this notion of convergence is invariant under dilations and translations parallel to the boundary), we find $\tilde r_k\rightharpoonup_{\text{symm}} 0$. Therefore, applying the definition of $S^*_{n,\sigma}(\R^n_+)$ to the sequence $(\tilde r_k/ I_{n,\sigma,\R^n_+}[\tilde r_k]^\frac12)$ and recalling that $M>0$, and therefore $T>0$, we obtain \eqref{eq:mmm3b}.

\emph{Step 6.} We are now ready to complete the proof. According to \eqref{eq:mmm1b}, \eqref{eq:mmm2b}, \eqref{eq:mmm3b} and the elementary inequality \eqref{eq:elem} we have
\begin{align*}
1 & = I_{n,\sigma,\R^n}[v] + T \\
& \geq I_{n,\sigma,\R^n}[v] + S^*_{n,\sigma}(\R^n_+) M^\frac{n-2\sigma}{n} \\
& = I_{n,\sigma,\R^n}[v] + S^*_{n,\sigma}(\R^n_+) \left( \left( S^*_{n,\sigma}(\R^n_+) \right)^{-\frac{n}{n-2\sigma}} - \int_{\R^n} |v|^\frac{2n}{n-2\sigma}\,\ud x \right)^\frac{n-2\sigma}{n} \\
& \geq I_{n,\sigma,\R^n}[v] + 1 - S^*_{n,\sigma}(\R^n_+) \left( \int_{\R^n} |v|^\frac{2n}{n-2\sigma}\,\ud x \right)^\frac{n-2\sigma}{n} \,.
\end{align*}
Thus, we have shown that
$$
I_{n,\sigma,\R^n}[v] \leq S^*_{n,\sigma}(\R^n_+) \left( \int_{\R^n} |v|^\frac{2n}{n-2\sigma}\,\ud x \right)^\frac{n-2\sigma}{n} \,.
$$
Bounding the left side from below by $S_{n,\sigma}(\R^n) \left( \int_{\R^n} |v|^\frac{2n}{n-2\sigma}\,\ud x \right)^\frac{n-2\sigma}{n}$ and recalling that $v\not\equiv 0$ we obtain the claimed inequality.


\subsection{Proof of Proposition \ref{mainhsprop}. Part 2}

We briefly sketch the proof of the reverse inequality
$$
S^*_{n,\sigma}(\R^n_+) \leq S_{n,\sigma}(\R^n) \,.
$$
Let $0\not\equiv w\in\mathring H^\sigma(\R^n)$ with compact support and let $(c_k)\subset\R$ be a sequence with $c_k\to\infty$. Let $\tilde w_k(x)=w(x',x_n-c_k)$. These functions belong to $\mathring H^\sigma[\R^n_+]$ for all sufficiently large $k$ and by an argument as in Step 3 of the previous proof we see that
$$
I_{n,\sigma,\R^n_+}[\tilde w_k] \to I_{n,\sigma,\R^n}[w] \,.
$$
Moreover, by Lemma \ref{convmodsymm} we know that $\tilde w_k\rightharpoonup_{\mathrm{symm}} 0$, and therefore also $\tilde w_k/I_{n,\sigma,\R^n_+}[\tilde w_k]^{1/2}\rightharpoonup_{\mathrm{symm}} 0$. Thus,
$$
S_{n,\sigma}^*(\R^n_+) \leq \liminf_{k\to\infty} \frac{I_{n,\sigma,\R^n_+}[\tilde w_k]}{\left( \int_{\R^n_+} |\tilde w_k|^\frac{2n}{n-2\sigma} \,\ud x\right)^\frac{n-2\sigma}{n}} = \frac{I_{n,\sigma,\R^n}[w]}{\left( \int_{\R^n} |w|^\frac{2n}{n-2\sigma} \,\ud x\right)^\frac{n-2\sigma}{n}} \,.
$$
Taking the infimum over all $w$ we obtain the claimed inequality.


\section{Bound on the half-space minimizer}
\label{sec:estimateofminimizer}

Our goal in this section is to prove Proposition \ref{prop:extremalfunct}.


\subsection{Reduction to a local bound}\label{sec:reduction}

The first step in the proof of Proposition \ref{prop:extremalfunct} is to reduce the global bound to a local statement. This argument is based on the invariance of equation \eqref{eq:minimizer} under inversion in a sphere. This inversion invariance of the minimization problem, which leads to the invariance of the Euler--Lagrange equation \eqref{eq:minimizer}, was already crucially used in Lieb's work \cite{Lieb} in the dual form of the Hardy--Littlewood--Sobolev inequality. In the local case $\sigma=1$ it appears famously in \cite{CaGiSp} and has also been used before in the case of the fractional Laplacian.

For $r>0$ we shall use the notation
$$
B_r^+ = \{ x\in\R^n_+:\ |x|< r\} \,.
$$
The local bound that we will prove in this section is

\begin{proposition}\label{prop:extremalfunctloc}
Let $n\geq 2$ and $1/2<\sigma<1$. Let $0\not\equiv u\in\mathring H^\sigma(\R^n_+)$ be non-negative and satisfy \eqref{eq:minimizer}. Then there are $0<c\leq C<\infty$ such that
$$
c x_n^{2\sigma-1} \leq u(x) \leq C x_n^{2\sigma-1}
\qquad\text{for}\ x\in B_1^+ \,.
$$
\end{proposition}

Clearly, by translation and dilation invariance a similar bound holds for half-balls of any radius centered at any point in $\R^{n-1}\times\{0\}$.

Accepting this proposition for the moment we now give the

\begin{proof}[Proof of Proposition \ref{prop:extremalfunct}] Let $u\in \mathring H^\sigma(\R^n_+)$ be a non-negative solution of \eqref{eq:minimizer}. Then by Proposition \ref{prop:extremalfunctloc} there are $0<c\leq C<\infty$ such that
$$
c x_n^{2\sigma-1} \leq u(x) \leq C x_n^{2\sigma-1}
\qquad\text{for all}\ x\in B_1^+ \,.
$$
On the other hand, $u_1(y)=|y|^{2\sigma-n} u(y/|y|^2) $ is also a solution of \eqref{eq:minimizer}. Thus, again by Proposition \ref{prop:extremalfunctloc} there are $0<c'\leq C'<\infty$ such that $c' y_n^{2\sigma-1}\leq u_1(y) \leq C' y_n^{2\sigma-1}$ for all $y\in B_1^+$. This means
$$
c' x_n^{2\sigma -1} |x|^{-n+2-2\sigma} \leq u(x) \leq C' x_n^{2\sigma -1} |x|^{-n+2-2\sigma}
\qquad\text{for all}\ x\in\R^n_+\setminus B_1^+ \,.
$$
Combining the two bounds we obtain the proposition.
\end{proof}


\subsection{Green's function bound}

For the proof of Proposition \ref{prop:extremalfunctloc} we need a bound on the Green's function of the fractional Dirichlet Laplacian $L_\sigma$. This operator is defined as an operator from $\mathring H^\sigma(\R^n_+)$ to the dual space $\mathring H^\sigma(\R^n_+)'$ by
$$
\left( L_\sigma u\right)(x) :=2a_{n,\sigma}\int_{\R^n_+} \frac{u(x)-u(y)}{|x-y|^{n+2\sigma}}\,\ud y \,.
$$
Here
$$
a_{n,\sigma} = 2^{2\sigma-1} \pi^{-\frac n2} \frac{\Gamma(\tfrac{n+2\sigma}{2})}{|\Gamma(-\sigma)|}
$$
is a positive constant which is chosen such that $a_{n,\sigma} I_{n,\sigma,\R^n}[u]= \int |\xi|^{2\sigma} |\hat u(\xi)|^2\,\ud\xi$ for $u\in \mathring H^\sigma(\R^n)$. We emphasize that $L_\sigma$ is \emph{not} the power $\sigma$ of the Dirichlet Laplacian on $\R^n_+$, see, e.g., \cite{Fr}. It will be important for us that the inverse of $L_\sigma$ is an integral operator on whose integral kernel $G$, the Green's function, we have two-sided bounds. These bounds are due to Chen--Kim--Song \cite{ChKiSo} and we are grateful to Prof. Z.-Q. Chen for discussions on them.

\begin{proposition}\label{green}
Let $n\geq 2$ and $1/2<\sigma<1$. Then there are constants $0<c\leq C<\infty$ such that the Green's function $G$ of $L_\sigma$ satisfies for all $x,y\in\R^n_+$
$$
\frac{c}{|x-y|^{n-2\sigma}} \min\left\{1,~ \frac{x_n^{2\sigma-1}y_n^{2\sigma-1}}{|y-x|^{4\sigma-2}} \right\}\le G(x,y)\le 
\frac{C}{|x-y|^{n-2\sigma}} \min\left\{1,~ \frac{x_n^{2\sigma-1}y_n^{2\sigma-1}}{|y-x|^{4\sigma-2}} \right\} \,.
$$
Moreover, the Green's function $\tilde G$ of $L_\sigma+1$ satisfies for all $x,y\in\R^n_+$
$$
\tilde G(x,y)\le 
\frac{C}{|x-y|^{n-2\sigma}} \min\left\{1,~ \frac{1}{|y-x|^{4\sigma}} \right\} \,.
$$
\end{proposition}

The bound on $\tilde G$ is not optimal but enough for our purposes. With little more effort the following proof would also provide an optimal two-sided bound.

\begin{proof}
The integral kernel $K_t$ of $e^{-tL_\sigma}$ satisfies for all $t>0$, $x,y\in\R^n_+$,
\begin{align*}
& \frac{c}{t^{\frac d{2\sigma}}} \min\left\{ 1,\frac{t^\frac{d+2\sigma}{2\sigma}}{|x-y|^{d+2\sigma}} \right\} \min\left\{ 1,\frac{x_n^{2\sigma-1}}{t^\frac{2\sigma-1}{2\sigma}} \right\} \min\left\{ 1,\frac{y_n^{2\sigma-1}}{t^\frac{2\sigma-1}{2\sigma}} \right\} \\
& \qquad \leq K_t(x,y) \leq \frac{C}{t^{\frac d{2\sigma}}} \min\left\{ 1,\frac{t^\frac{d+2\sigma}{2\sigma}}{|x-y|^{d+2\sigma}} \right\} \min\left\{ 1,\frac{x_n^{2\sigma-1}}{t^\frac{2\sigma-1}{2\sigma}} \right\} \min\left\{ 1,\frac{y_n^{2\sigma-1}}{t^\frac{2\sigma-1}{2\sigma}} \right\}.
\end{align*}
This bound is stated in \cite[Thm. 1.1]{ChKiSo} for $t\in [0,T]$ and $x,y\in\R^n_+$ with constants $c$ and $C$ depending on $T$, but by scaling the variables $x,y$ and $t$ the bound for $t=1$ implies the bound for general $t>0$. Since $G(x,y) = \int_0^\infty K_t(x,y)\,\ud t$, we obtain the bounds on $G$ by tedious, but elementary computations as in the proof of \cite[Cor. 1.2]{ChKiSo}. Similarly, $\tilde G(x,y) = \int_0^\infty e^{-t} K_t(x,y)\,\ud t$ and using $K_t(x,y)\leq C t^{-\frac{d}{2\sigma}} \min\{1, t^\frac{d+2\sigma}{2\sigma}|x-y|^{-d+2\sigma}\}$ we obtain the bound on $\tilde G$.
\end{proof}


\subsection{Proof of Proposition \ref{prop:extremalfunctloc}}

We write equation \eqref{eq:minimizer} as
$$
L_\sigma u = a_{n,\sigma} u^\frac{n+2\sigma}{n-2\sigma} \,.
$$
Since $u\in\mathring H^\sigma(\R^n_+)$ we have, by the Sobolev inequality, $u\in L^\frac{2n}{n-2\sigma}(\R^n_+)$, and therefore $u^\frac{n+2\sigma}{n-2\sigma}\in L^\frac{2n}{n+2\sigma}(\R^n_+)$. Since $0\leq G(x,y) \leq C |x-y|^{-n+2\sigma}$ by Proposition \ref{green}, the Hardy--Littlewood--Sobolev inequality implies that $L_\sigma^{-1} u^\frac{n+2\sigma}{n-2\sigma}$ is well-defined and belongs to $L^\frac{2n}{n-2\sigma}(\R^n_+)$. In this way, the equation becomes
\begin{equation}
\label{eq:inter-minimizer}
u(x) = a_{n,\sigma} \int_{\R^n_+} G(x,y) u(y)^\frac{n+2\sigma}{n-2\sigma}\,\ud y 
\qquad\text{for}\ x\in\R^n_+ \,.
\end{equation}

\emph{Step 1. Upper bound.}
We will show that
\begin{equation}
\label{eq:boundedness1}
u(x) \leq \frac{C}{(1+|x|)^{n+2\sigma}}
\qquad\text{for all}\ x\in\R^n \,.
\end{equation}
Indeed, once this is proved, we can combine it with \eqref{eq:inter-minimizer} and the bound on $G$ from Proposition \ref{green} to obtain
\[
u(x) \le C \int_{\R^n_+}\frac{x_n^{2\sigma-1}y_n^{2\sigma-1}}{|y-x|^{n+2\sigma-2}}\frac{1}{(1+|y|)^{n+2\sigma}}\,\ud y\le C' x_n^{2\sigma-1} \,,
\]
which is the claimed upper bound.

It remains to prove \eqref{eq:boundedness1}. Proceeding as in the derivation of \eqref{eq:inter-minimizer} we obtain
\begin{equation*}
u(x) = \int_{\R^n_+} \tilde G(x,y) \left(a_{n,\sigma} u(y)^\frac{n+2\sigma}{n-2\sigma}+ u(y)\right)\ud y 
\qquad\text{for}\ x\in\R^n_+ \,.
\end{equation*}
We use the bound on $\tilde G$ from Proposition \ref{green} and obtain
\begin{equation}
\label{eq:inter-minimizer2}
u(x) \leq h + Y*(Au)
\end{equation}
where
$$
h(x) := \int_{\{ u \leq 1\}} \tilde G(x,y) \left(a_{n,\sigma} u(y)^\frac{n+2\sigma}{n-2\sigma}+ u(y)\right)\ud y \,,
$$
and
$$
Y(x) := C |x|^{-d+2\sigma} \min\{ 1, |x|^{-4\sigma}\}
\qquad\text{and}\qquad
A(x) := \chi_{\{u\geq 1\}} \left( a_{n,\sigma} u(x)^\frac{4\sigma}{n-2\sigma} -1 \right) .
$$
Since $0\leq h \leq (a_{n,\sigma}+1) \int_{\R^n} Y(x)\,\ud x$ and $Y\in L^1(\R^n)$, we have $h\in L^\infty(\R^n)$. Moreover, $Y\in L^\frac{n}{n-2\sigma}_{\rm weak}(\R^n)$, $A\in L^\frac n{2\sigma}(\R^n)$ and $A$ has support of finite measure. Thus, \cite[Lem. A.1]{BrLi2} implies that $\int_{\{u\geq 1\}} u^q\,\ud x <\infty$ for any $q<\infty$. In particular, $Au\in L^p(\R^n)$ for some $p>\frac{n}{2\sigma}$ and therefore $Y*(Au)\in L^\infty(\R^n)$. In view of \eqref{eq:inter-minimizer2} we obtain $u\in L^\infty(\R^n)$ and, in particular, \eqref{eq:boundedness1} for $|x|\leq 1$. (In passing, we note that instead of \cite{BrLi2} we could also have used \cite[Cor. 1.1]{Li}.)

Similar as in Subsection \ref{sec:reduction} we note that $u_1(y)=|y|^{2\sigma-n} u(y/|y|^2)$ is also a solution of \eqref{eq:minimizer} and therefore, by the above argument $u_1\in L^\infty(\R^n)$. This proves \eqref{eq:boundedness1} for $|x|\geq 1$ and therefore concludes the proof of \eqref{eq:boundedness1}.

\emph{Step 2. Lower bound.}
Since $u$ is continuous in $\R^n_+$ (this follows easily from the equation since we have already shown that $u\in L^\infty(\R^n_+)$) and positive by the maximum principle, there is a $c>0$ such that $u\geq c$ in $B_1(3e_n)$. Thus, by the lower bound on $G$ in Proposition \ref{green},
\[
u(x)\ge c \int_{B_1(3e_n)} G(x,y)\,\ud y\ge c' x_n^{2\sigma-1}\quad\mbox{for }x\in B_1^+.
\]


\section{Verifying the strict inequality \eqref{eq:mainbddass}}

Our goal in this section is to prove Theorem \ref{mainbddbinding}. The overall proof strategy resembles that of Theorem \ref{mainhsbinding}, with the important difference, however, that the minimizer of $S_{n,\sigma}(\R^n_+)$ is not known explicitly. We therefore begin by collecting some facts about this function.

We are assuming that the infimum $S_{n,\sigma}(\R^n_+)$ is attained. Let $\Theta$ be a minimizer, normalized so that $\|\Theta\|_{L^\frac{2n}{n-2\sigma}(\R^n_+)}=1$. As discussed before the statement of Proposition \ref{prop:extremalfunct} we may assume that $\Theta$ is non-negative. For $\lambda>0$ let
$$
\Theta_\lambda(x) = \lambda^\frac{n-2\sigma}{2} \Theta(\lambda x) \,.
$$
The Euler--Lagrange equation of the minimization problem reads
$$
2 \int_{\R^n_+} \frac{\Theta_\lambda(x)-\Theta_\lambda(y)}{|x-y|^{n+2\sigma}} \,\ud y = S_{n,\sigma}(\R^n_+) \Theta_\lambda(x)^\frac{n+2\sigma}{n-2\sigma} \,.
$$
In the following proof we will need the following estimates which follow from the bound of Proposition \ref{prop:extremalfunct} (which can be applied to a multiple of $\Theta$. For any $p>\frac{n}{n-1}$ we have
\begin{equation}
\label{eq:thetacomp1}
\int_{\R^n_+} \Theta_\lambda^p \,\ud x = C_p \lambda^{-n+ \frac{p(n-2\sigma)}{2}}
\qquad\text{with}\ C_p = \int_{\R^n_+} \Theta^p \,\ud x <\infty
\end{equation}
and, for any fixed $0<\rho<\infty$,
\begin{equation}
\label{eq:thetacomp2}
\int_{\R^n_+\setminus B_\rho} \Theta_\lambda^p \,\ud x = O( \lambda^{-\frac{p(n+2\sigma-2)}{2}}) \,.
\end{equation}
(To prove the latter bound we split the region of integration into the sets where $x_n\geq |x'|$ and where $x_n<|x'|$.) Moreover, for $0<p<\frac{n}{n-1}$ and again for any fixed $0<\rho<\infty$,
\begin{equation}
\label{eq:thetacomp4}
\int_{B_\rho} \Theta_\lambda^p \,\ud x = O( \lambda^{-\frac{p(n+2\sigma-2)}{2}}) \,.
\end{equation}
Finally, for any fixed $0<\rho<\infty$,
\begin{equation}
\label{eq:thetacomp3}
\int_{\R^n_+\setminus B_\rho} \frac{\Theta_\lambda}{|x|^{n+2\sigma}} \,\ud x = O( \lambda^{-\frac{n+2\sigma-2}{2}}) \,.
\end{equation}

After these preliminaries we are ready to give the

\begin{proof}[Proof of Theorem \ref{mainbddbinding}]
Since the assumption and the conclusion of the proposition are invariant under translations and dilations of $\Omega$, we may assume that
$$
B_4^+ \subset \Omega \,.
$$
Let $\eta$ be a cut-off function as in the proof of Proposition \ref{stricths} and put
$$
\theta_\lambda = \eta \Theta_\lambda \,.
$$
This function belongs to $C^1_c(\Omega)$ and we will estimate $\|\theta_\lambda\|_\frac{2n}{n-2\sigma}$ and $I_{n,\sigma,\Omega}[\theta_\lambda]$ as $\lambda\to\infty$.

By the normalization of $\Theta_\lambda$, \eqref{eq:thetacomp1} and \eqref{eq:thetacomp2} we obtain
\be\label{eq:up-Error1}
\int_{\Omega} \theta_\lda^{\frac{2n}{n-2\sigma}} \,\ud x= 1- O(\lda^{-\frac{n(n+2\sigma-2)}{n-2\sigma}}) \,.
\ee

In order to bound $I_{n,\sigma,\Omega}[\theta_\lambda]$ we write
\begin{equation}
\label{eq:up-Error2a}
I_{n,\sigma,\Omega}[\theta_\lambda] = \int_\Omega \theta_\lambda(x) g_\lambda(x)\,\ud x
\end{equation}
with
$$
g_\lambda(x) := 2 \int_\Omega \frac{\theta_\lambda(x)-\theta_\lambda(y)}{|x-y|^{n+2\sigma}}\,\ud y
$$
and estimate $g_\lambda$ pointwise in $B_1^+$ and $B_3^+\setminus B_1^+$.

For $x\in B_{1}^+$, using the fact that $\Omega\subset\R^n_+$ and \eqref{eq:thetacomp3},
\be \label{eq:up-Error2}
\begin{split}
g_\lambda(x) & = 2 \int_{\R^n_+} \frac{\Theta_\lda(x)-\Theta_\lda(y)}{|x-y|^{n+2\sigma}}\,\ud y
+  2 \int_{\R^n_+} \frac{\Theta_\lda(y)-\theta_\lda (y)}{|x-y|^{n+2\sigma}}\,\ud y
-2 \int_{\R^n_+\setminus \om} \frac{\Theta_\lda(x)}{|x-y|^{n+2\sigma}}\,\ud y\\
& =S_{n,\sigma}(\R^n_+) \Theta_{\lda}(x)^{\frac{n+2\sigma}{n-2\sigma}} 
+ 2\int_{\R^n_+\setminus B_2^+} \frac{\Theta_\lda(y)-\theta_\lda (y)}{|x-y|^{n+2\sigma}}\,\ud y + W_\Omega(x)\Theta_\lda(x)\\
& \leq S_{n,\sigma}(\R^n_+) \Theta_{\lda}(x)^{\frac{n+2\sigma}{n-2\sigma}} 
+ 2\int_{\R^n_+\setminus B_2^+} \frac{\Theta_\lda(y)}{|x-y|^{n+2\sigma}}\,\ud y + W_\Omega(x)\Theta_\lda(x)\\
& = S_{n,\sigma}(\R^n_+) \Theta_{\lda}(x)^{\frac{n+2\sigma}{n-2\sigma}} +O(\lda^{-\frac{n+2\sigma-2}{2}}) + W_\Omega(x) \Theta_\lda(x)
\end{split}
\ee
with
$$
W_\Omega(x) := -2 \int_{\R^n_+\setminus\Omega} \frac{\ud y}{|x-y|^{n+2\sigma}} \,.
$$
Since $\R^n_+\setminus\Omega$ contains an interior point, we have $W_\Omega(x) \leq -\epsilon_0<0$ for all $x\in B_1^+$.

Next, let $x\in B_3^+\setminus B_1^+$ and write
$$
g_\lambda(x) = 2 \eta(x) \int_{\om}  \frac{\Theta_\lda(x)-\Theta_\lda(y)}{|x-y|^{n+2\sigma}}\,\ud y + 2 \int_{\om}  \frac{(\eta(x)-\eta(y))\Theta_\lda(y)}{|x-y|^{n+2\sigma}}\,\ud y \,.
$$
We have, using $B_4^+\subset\Omega\subset\R^n_+$ and \eqref{eq:thetacomp3},
\begin{align*}
2 \int_\Omega \frac{\Theta_\lda(x)-\Theta_\lda(y)}{|x-y|^{n+2\sigma}}\,\ud y & = S_{n,\sigma}(\R^n_+) \Theta_\lambda(x)^\frac{n+2\sigma}{n-2\sigma} - 2 \int_{\R^n_+\setminus\om}  \frac{\Theta_\lda(x)-\Theta_\lda(y)}{|x-y|^{n+2\sigma}}\,\ud y \\
& \leq S_{n,\sigma}(\R^n_+) \Theta_\lambda(x)^\frac{n+2\sigma}{n-2\sigma} + 2 \int_{\R^n_+\setminus\om}  \frac{\Theta_\lda(y)}{|x-y|^{n+2\sigma}}\,\ud y \\
& \leq S_{n,\sigma}(\R^n_+) \Theta_\lambda(x)^\frac{n+2\sigma}{n-2\sigma} + C \int_{\R^n_+\setminus\om}  \frac{\Theta_\lda(y)}{|y|^{n+2\sigma}}\,\ud y \\
& \leq S_{n,\sigma}(\R^n_+) \Theta_\lambda(x)^\frac{n+2\sigma}{n-2\sigma} + O(\lambda^{-\frac{n+2\sigma-2}{2}}) \,.
\end{align*}

Moreover,
\begin{align*}
2 \int_{\om}  \frac{(\eta(x)-\eta(y))\Theta_\lda(y)}{|x-y|^{n+2\sigma}}\,\ud y
& =  2 \int_{B_1^+}  \frac{(\eta(x)-\eta(y))\Theta_\lda(y)}{|x-y|^{n+2\sigma}}\,\ud y \\
& \qquad + 2 \int_{\om\setminus B_4^+}  \frac{(\eta(x)-\eta(y))\Theta_\lda(y)}{|x-y|^{n+2\sigma}}\,\ud y \\
& \qquad + 2 \int_{B_4^+\setminus B_1^+}  \frac{(\eta(x)-\eta(y))\Theta_\lda(y)}{|x-y|^{n+2\sigma}}\,\ud y \,.
\end{align*}
Let us discuss the three terms on the right side separately. We have
$$
2 \int_{B_1^+}  \frac{(\eta(x)-\eta(y))\Theta_\lda(y)}{|x-y|^{n+2\sigma}}\,\ud y = 0
\qquad\text{if}\ x\in B_2^+
$$
and, if $x\in B_3^+\setminus B_2^+$, by \eqref{eq:thetacomp4},
$$
2 \int_{B_1^+}  \frac{(\eta(x)-\eta(y))\Theta_\lda(y)}{|x-y|^{n+2\sigma}}\,\ud y \leq 2 \int_{B_1^+}  \frac{\Theta_\lda(y)}{|x-y|^{n+2\sigma}}\,\ud y \leq C \int_{B_1^+}  \Theta_\lda(y)\,\ud y = O( \lambda^{-\frac{n+2\sigma-2}{2}}) \,. 
$$
We have
$$
2 \int_{\om\setminus B_4^+}  \frac{(\eta(x)-\eta(y))\Theta_\lda(y)}{|x-y|^{n+2\sigma}}\,\ud y
\leq 2 \int_{\om\setminus B_4^+}  \frac{\Theta_\lda(y)}{|x-y|^{n+2\sigma}}\,\ud y
\leq C \int_{\R^n_+\setminus B_4^+}  \frac{\Theta_\lda(y)}{|y|^{n+2\sigma}}\,\ud y = O( \lambda^{-\frac{n+2\sigma-2}{2}}) \,. 
$$
Finally, we estimate the last term in an integral sense. We have
\begin{align*}
& 2 \int_{B_3^+\setminus B_1^+} \theta_\lambda(x) \int_{B_4^+\setminus B_1^+}  \frac{(\eta(x)-\eta(y))\Theta_\lda(y)}{|x-y|^{n+2\sigma}}\,\ud y \,\ud x \\
& \qquad = 2 \int_{B^4\setminus B_1^+} \Theta_\lambda(x) \eta(x) \int_{B_4^+\setminus B_1^+}  \frac{(\eta(x)-\eta(y))\Theta_\lda(y)}{|x-y|^{n+2\sigma}}\,\ud y \,\ud x \\
& \qquad = \iint_{(B_4^+\setminus B_1^+)\times(B_4^+\setminus B_1^+)}  \frac{(\eta(x)-\eta(y))^2 \Theta_\lambda(x) \Theta_\lda(y)}{|x-y|^{n+2\sigma}}\,\ud y \,\ud x \\
& \qquad \leq C \iint_{(B_4^+\setminus B_1^+)\times(B_4^+\setminus B_1^+)}  \frac{\Theta_\lambda(x) \Theta_\lda(y)}{|x-y|^{n+2\sigma-2}}\,\ud y \,\ud x \,.
\end{align*}
We now use Proposition \ref{prop:extremalfunct} to bound $\Theta(x) \lesssim |x|^{-n+1}$ and obtain
\begin{align*}
& \iint_{(B_4^+\setminus B_1^+)\times(B_4^+\setminus B_1^+)}  \frac{\Theta_\lambda(x) \Theta_\lda(y)}{|x-y|^{n+2\sigma-2}}\,\ud y \,\ud x \\
& \qquad \lesssim \lambda^{-n-2\sigma+2} \iint_{(B_4^+\setminus B_1^+)\times(B_4^+\setminus B_1^+)}  \frac{\ud y \,\ud x}{|x|^{n-1} |x-y|^{n+2\sigma-2} |y|^{n-1}} \\
& \qquad \lesssim \lambda^{-n-2\sigma+2} \int_{B_4^+\setminus B_1^+}  \frac{\ud x}{|x|^{2n+2\sigma-4}} \\
& \qquad \lesssim \lambda^{-n-2\sigma+2} \,.
\end{align*}

To summarize we have shown that for $x\in B_3^+\setminus B_1^+$ we have
\begin{equation}
\label{eq:up-Error3}
g_\lambda(x) \leq S_{n,\sigma}(\R^n_+) \eta(x) \Theta_\lambda(x)^\frac{n+2\sigma}{n-2\sigma} + O(\lambda^{-\frac{n+2\sigma-2}{n}}) + \tilde g_\lambda(x)
\end{equation}
with
$$
\int_{B_3^+\setminus B_1^+} \theta_\lambda(x) \tilde g_\lambda(x) \,\ud x = O(\mathcal \lambda^{-n-2\sigma+2} )
$$

Let us insert the bounds \eqref{eq:up-Error2} and \eqref{eq:up-Error3} into \eqref{eq:up-Error2a}. We obtain with the help of \eqref{eq:thetacomp4}
\begin{align*}
I_{n,\sigma,\Omega}[\theta_\lambda] & \leq \int_{B_1^+} \Theta_\lambda(x) \left( S_{n,\sigma}(\R^n_+) \Theta_\lambda(x)^\frac{n+2\sigma}{n-2\sigma} + O(\lambda^{-\frac{n+2\sigma-2}{2}}) -\epsilon_0 \Theta_\lambda(x) \right) \ud x \\
& \qquad + \int_{B_3^+\setminus B_1^+} \theta_\lambda(x) \left( S_{n,\sigma}(\R^n_+) \Theta_\lambda(x)^\frac{n+2\sigma}{n-2\sigma} + O(\lambda^{-\frac{n+2\sigma-2}{2}} \right) \ud x \\
& \qquad + O(\lambda^{-n-2\sigma+2}) \\
& \leq S_{n,\sigma}(\R^n_+) \int_{B_3^+} \Theta_\lambda(x)^\frac{2n}{n-2\sigma} \,\ud x - \epsilon_0 \int_{B_1^+} \Theta_\lambda(x)^2 \,\ud x + O(\lambda^{-n-2\sigma+2}) \,.
\end{align*}
By \eqref{eq:thetacomp1} and the normalization of $\Theta$, we have
$$
\int_{B_3^+} \Theta_\lambda(x)^\frac{2n}{n-2\sigma} \,\ud x \leq 1 \,.
$$
Next, if $n\geq 3$ we have $2>\frac{n}{n-1}$ (this is where the assumption $n\neq 2$ enters!), and therefore by \eqref{eq:thetacomp1} and \eqref{eq:thetacomp2}
$$
\int_{B_1^+} \Theta_\lambda(x)^2 \,\ud x = C_2 \lambda^{-2\sigma} +  O(\lambda^{-n-2\sigma+2}) \,.
$$
Combining these bounds with \eqref{eq:up-Error1} we finally obtain
\begin{align*}
S_{n,\sigma}(\Omega) & \leq \frac{I_{n,\sigma,\Omega}[\theta_\lambda]}{\left( \int_\Omega \theta_\lambda^\frac{2n}{n-2\sigma}\,\ud x\right)^\frac{n-2\sigma}{n}} \\
& \leq \left(1-C \lambda^{-\frac{n(n+2\sigma-2)}{n-2\sigma}} \right) \left( S_{n,\sigma}(\R^n_+) - C_2^{-1} \epsilon_0 \lambda^{-2\sigma} + O(\lambda^{-n-2\sigma+2}) \right).
\end{align*}
The right side is strictly less than $S_{n,\sigma}(\R^n_+)$ provided that $\lambda$ is sufficiently large. This completes the proof of the theorem for $n\geq 3$.

Finally, let $n=2$. (Note that we assume $\sigma>1/2$ and so the case $n=2$ might be void if it is true that $S_{n,\sigma}(\R^n_+)$ is not attained for $n< 4\sigma$.) The proof is essentially the same as for $n\geq 3$, except that we use the \emph{lower} bound from Proposition \ref{prop:extremalfunct} to deduce that $\int_{B_1^+} \Theta_\lambda(x)^2\,\ud x \geq c \lambda^{-2\sigma}\ln\lambda$. Thus, we have again $S_{n,\sigma}(\Omega)< S_{n,\sigma}(\R^2_+)$, as claimed.
\end{proof}


\section{The case of a bounded domain}

\subsection{Outline of the strategy}

In this section we prove Theorem \ref{mainbdd}. We set
$$
S^{**}_{n,\sigma}(\Omega) := \inf \left\{ \liminf_{k\to\infty} \left( \int_{\Omega} |u_k|^\frac{2n}{n-2\sigma} \,\ud x \right)^{-\frac{n-2\sigma}{n}} :\ I_{n,\sigma,\Omega}[u_k] =1 \,,\ u_k \rightharpoonup 0 \ \text{in}\ \mathring H^\sigma(\Omega) \right\} \,.
$$
The key step in the proof of Theorem \ref{mainbdd} is the computation of this number.

\begin{proposition}\label{thresholdbdd}
Let $0<\sigma<1/2$ if $n=1$ and $0<\sigma<1$ if $n\geq 2$. Let $\Omega\subset\R^n$ be a bounded set with $C^1$ boundary. Then
$$
S^{**}_{n,\sigma}(\Omega) = S_{n,\sigma}(\R^n_+) \,.
$$
\end{proposition}

Accepting this proposition for the moment we give the

\begin{proof}[Proof of Theorem \ref{mainbdd}]
Let $(u_k)$ be a minimizing sequence for $S_{n,\sigma}(\Omega)$, which is normalized in $\mathring H^\sigma(\Omega)$. Then, up to passing to a subsequence, $u_k \rightharpoonup u$ in $\mathring H^\sigma(\Omega)$. Moreover, by Rellich's theorem after passing to another subsequence if necessary, $u_k\to u$ almost everywhere. Let $r_k:=u_k-u$. Then, by weak convergence in $\mathring H^\sigma(\Omega)$,
$$
1 = I_{n,\sigma,\Omega}[u_k] = I_{n,\sigma,\Omega}[u] + I_{n,\sigma,\Omega}[r_k] + o(1) \,.
$$
Thus, $I_{n,\sigma,\Omega}[r_k]$ converges and
\begin{equation}
\label{eq:mmm1}
T := \lim_{k\to\infty} I_{n,\sigma,\Omega}[r_k]
\qquad\text{satisfies}\qquad
1= I_{n,\sigma,\Omega}[u] + T \,.
\end{equation}
Moreover, by almost everywhere convergence and the Br\'ezis--Lieb lemma \cite{BrLi},
$$
S_{n,\sigma}(\Omega)^{-\frac{n-2\sigma}{n}} + o(1) = \int_\Omega |u_k|^\frac{2n}{n-2\sigma}\,\ud x = \int_\Omega |u|^\frac{2n}{n-2\sigma}\,\ud x + \int_\Omega |r_k|^\frac{2n}{n-2\sigma}\,\ud x + o(1) \,. 
$$
Thus, $\int_\Omega |r_k|^\frac{2n}{n-2\sigma}\,\ud x$ converges and
\begin{equation}
\label{eq:mmm2}
M := \lim_{k\to\infty} \int_\Omega |r_k|^\frac{2n}{n-2\sigma}\,\ud x 
\qquad\text{satisfies}\qquad
S_{n,\sigma}(\Omega)^{-\frac{n-2\sigma}{n}}  = \int_\Omega |u|^\frac{2n}{n-2\sigma}\,\ud x + M \,.
\end{equation}
Moreover, by Proposition \ref{thresholdbdd},
\begin{equation}
\label{eq:mmm3}
T \geq S_{n,\sigma}(\R^n_+)\ M^\frac{n-2\sigma}{n} \,.
\end{equation}
(Here one distinguishes the cases $M=0$, where the inequality is trivial, and $M>0$, where one can apply the definition of $S^{**}_{n,\sigma}(\Omega)$.)

Given \eqref{eq:mmm1}, \eqref{eq:mmm2} and \eqref{eq:mmm3} the proof is concluded by the same arguments as before. We use the elementary inequality \eqref{eq:elem} with $\theta=(n-2\sigma)/n$ and find
\begin{align*}
1 & = I_{n,\sigma,\Omega}[u] + T \\
& \geq I_{n,\sigma,\Omega}[u] + S_{n,\sigma}(\R^n_+) M^\frac{n-2\sigma}{n} \\
& = I_{n,\sigma,\Omega}[u] + \left( S_{n,\sigma}(\R^n_+) - S_{n,\sigma}(\Omega) \right) M^\frac{n-2\sigma}{n} + S_{n,\sigma}(\Omega) \left(S_{n,\sigma}(\Omega)^{-\frac{n-2\sigma}{n}} - \int_\Omega |u|^\frac{2n}{n-2\sigma}\,\ud x \right)^\frac{n-2\sigma}{n} \\
& \geq I_{n,\sigma,\Omega}[u] + \left( S_{n,\sigma}(\R^n_+) - S_{n,\sigma}(\Omega) \right) M^\frac{n-2\sigma}{n} + 1 - S_{n,\sigma}(\Omega) \left( \int_\Omega |u|^\frac{2n}{n-2\sigma}\,\ud x \right)^\frac{n-2\sigma}{n} \,.
\end{align*}
Thus, we have shown that
$$
I_{n,\sigma,\Omega}[u] - S_{n,\sigma}(\Omega) \left( \int_\Omega |u|^\frac{2n}{n-2\sigma}\,\ud x \right)^\frac{n-2\sigma}{n} + \left( S_{n,\sigma}(\R^n_+) - S_{n,\sigma}(\Omega) \right) M^\frac{n-2\sigma}{n} \leq 0 \,.
$$
Since $I_{n,\sigma,\Omega}[u] \geq S_{n,\sigma}(\Omega) \left( \int_\Omega |u|^\frac{2n}{n-2\sigma}\,\ud x \right)^\frac{n-2\sigma}{n}$ and since $S_{n,\sigma}(\R^n_+) > S_{n,\sigma}(\Omega)$ by assumption, we deduce that $M=0$, so $u\not\equiv 0$, and then
$$
I_{n,\sigma,\Omega}[u] \leq S_{n,\sigma}(\Omega) \left( \int_\Omega |u|^\frac{2n}{n-2\sigma}\,\ud x \right)^\frac{n-2\sigma}{n},
$$
implies that $u$ is an optimizer. Finally, inequality \eqref{eq:mmm3} must be an equality and therefore $T=0$. This means that $I_{n,\sigma,\Omega}[u]=1$ and therefore $(u_k)$ converges, in fact, \emph{strongly} in $\mathring H^\sigma(\Omega)$ to $u$. This completes the proof of the theorem.
\end{proof}

Thus, we are left with proving Proposition \ref{thresholdbdd}. For the proof of the inequality $S^{**}_{n,\sigma}(\Omega)\geq S_{n,\sigma}(\R^n_+)$ (which is the only thing needed in the proof of Theorem \ref{mainbdd}) we use the following bound.

\begin{proposition}\label{almostineq}
Let $0<\sigma<1/2$ if $n=1$ and $0<\sigma<1$ if $n\geq 2$. Let $\Omega\subset\R^n$ be a bounded set with $C^1$ boundary. For every $\epsilon>0$ there is a $C_\epsilon<\infty$ such that for all $u\in\mathring H^\sigma(\Omega)$
$$
I_{n,\sigma,\Omega}[u] \geq (1-\epsilon) S_{n,\sigma}(\R^n_+) \left( \int_\Omega |u|^\frac{2n}{n-2\sigma} \,\ud x \right)^\frac{n-2\sigma}{n} - C_\epsilon \int_\Omega |u|^2\,\ud x \,.
$$
\end{proposition}

Let us use this proposition to give the

\begin{proof}[Proof of Proposition \ref{thresholdbdd}. Part 1]
We prove that $S^{**}_{n,\sigma}(\Omega)\geq S_{n,\sigma}(\R^n_+)$. Let $(u_k)\subset\mathring H^\sigma(\Omega)$ with $I_{n,\sigma,\Omega}[u_k]=1$ and $u_k\rightharpoonup 0$ in $\mathring H^\sigma(\Omega)$. For any $\epsilon>0$ we have by Proposition~\ref{almostineq}
$$
I_{n,\sigma,\Omega}[u_k] \geq (1-\epsilon) S_{n,\sigma}(\R^n_+) \left( \int_\Omega |u_k|^\frac{2n}{n-2\sigma}\,\ud x \right)^\frac{n-2\sigma}{n} - C_\epsilon \int_\Omega |u_k|^2\,\ud x \,.
$$
By Rellich's theorem we have $u_k\to 0$ in $L^2(\Omega)$ and therefore
$$
1 \geq (1-\epsilon) S_{n,\sigma}(\R^n_+) \limsup_{k\to\infty} \left( \int_\Omega |u_k|^\frac{2n}{n-2\sigma}\,\ud x \right)^\frac{n-2\sigma}{n} \,.
$$
Since $\epsilon>0$ is arbitrary, we obtain
$$
\limsup_{k\to\infty} \left( \int_\Omega |u_k|^\frac{2n}{n-2\sigma}\,\ud x \right)^\frac{n-2\sigma}{n} \leq S_{n,\sigma}(\R^n_+)^{-1} \,,
$$
which is the claimed inequality.
\end{proof}


\subsection{Proof of Proposition \ref{almostineq}}

The proof of Proposition \ref{almostineq} is based on a simple straightening of the boundary, which appears in the proof of the following lemma. It essentially appears already as \cite[Lemma 14]{FrGe}, but we include the simple proof for the sake of completeness.

\begin{lemma}\label{changeofvar}
Let $0<\sigma<1/2$ if $n=1$ and $0<\sigma<1$ if $n\geq 2$. Let $\Omega\subset\R^n$ be a bounded set with $C^1$ boundary and let $\epsilon>0$.

\begin{enumerate}
\item There are $\delta>0$ and $C<\infty$ such that for every ball $B$ of radius $\delta$ centered at a point in $\partial\Omega$ and for every $u\in\mathring H^\sigma(\Omega)$ with support in $\overline B$ one has
$$
I_{n,\sigma,\Omega}[u] \geq (1-\epsilon) I_{n,\sigma,\R^n_+}[\tilde u] - C \int_\Omega |u|^2\,\ud x \,.
$$
Here $\tilde u\in\mathring H^\sigma(\R^n_+)$ and is obtained from $u$ by a change of variables with Jacobian equal to one.
\item There are $\delta>0$ and $C<\infty$ such that for every ball $B$ of radius $\delta$ centered at a point in $\partial\R^n_+$ and for every $v\in\mathring H^\sigma(\R^n_+)$ with support in $\overline B$ one has
$$
I_{n,\sigma,\Omega}[\tilde v] \leq (1+\epsilon) I_{n,\sigma,\R^n_+}[v] + C \int_{\R^n_+} |v|^2\,\ud x \,.
$$
Here $\tilde v\in\mathring H^\sigma(\Omega)$ and is obtained from $v$ by a change of variables with Jacobian equal to one.
\end{enumerate}
\end{lemma}

\begin{proof}
Let $B^*$ and $D^*$ be balls of radius $2\delta$ in $\R^n$ and $\R^{n-1}$, respectively, centered at the origin. After a translation and a rotation we may assume that
$$
\partial \Omega \cap B^* = \{ (x',x_n) :\ x'\in D^* \,,\ x_n =\phi(x')\} \cap B^*
$$
and
$$
\Omega \cap B^* = \{ (x',x_n) :\ x'\in D^* \,,\ x_n >\phi(x')\} \cap B^* \,,
$$
where $\phi: D^*\to\R$ is a $C^1$ function with $\phi(0)=0$ and $\nabla\phi(0)=0$. We change variables
$$
\xi_j = \Phi_j(x) = x_j \quad \text{if}\ 1\leq j\leq n-1
\qquad\text{and}\qquad
\xi_n = \Phi_n(x) = x_n - \phi(x') \,.
$$
Note that the Jacobian of $\Phi$ is equal to one. 

Given $\epsilon>0$ we choose $\delta>0$ so small that
$$
\sup_{x'\in D^*} |\nabla\phi(x')| + \sup_{x'\in D^*} |\nabla\phi(x')|^2 \leq (1-\epsilon)^{-\frac{2}{n+2\sigma}} - 1 \,,
\quad
\sup_{x'\in D^*} |\nabla\phi(x')| \leq 1- (1+\epsilon)^{-\frac{2}{n+2\sigma}} \,.
$$
This is possible since $\nabla\phi(0)=0$. (More precisely, the inequality should hold for any point on the boundary and any $\phi$ corresponding to that point. This is possible since the boundary is compact and all $\nabla\phi$'s can be controlled by a common modulus of continuity.)

We claim that we have
\begin{align*}
\left| \frac{1}{|\Phi^{-1}(\xi)-\Phi^{-1}(\eta)|^{n+2\sigma}} - \frac{1}{|\xi-\eta|^{n+2\sigma}} \right| \leq \frac{\epsilon}{|\xi-\eta|^{n+2\sigma}}
\qquad\text{for all}\ \xi',\eta'\in D^* \,.
\end{align*}
In fact, this is equivalent to
$$
\left| \left( \frac{(\xi'-\eta')^2 + (\xi_n-\eta_n)^2}{(\xi'-\eta')^2 + (\xi_n+\phi(\xi') -\eta_n - \phi(\eta'))^2} \right)^\frac{n+2\sigma}{2} - 1 \right| \leq \epsilon \,,
$$
which is the same as
$$
(1+\epsilon)^{-\frac{2}{n+2\sigma}} \leq \frac{(\xi'-\eta')^2 + (\xi_n+\phi(\xi') -\eta_n - \phi(\eta'))^2}{(\xi'-\eta')^2 + (\xi_n-\eta_n)^2} \leq (1-\epsilon)^{-\frac{2}{n+2\sigma}}
$$
or as
$$
-1+ (1+\epsilon)^{-\frac{2}{n+2\sigma}} \leq \frac{2(\xi_n-\eta_n)(\phi(\xi')-\phi(\eta')) + (\phi(\xi')-\phi(\eta'))^2}{(\xi'-\eta')^2 + (\xi_n-\eta_n)^2} \leq (1-\epsilon)^{-\frac{2}{n+2\sigma}} - 1.
$$
Since $|\phi(\xi')-\phi(\eta')|\leq (\sup_D |\nabla\phi|) |\xi'-\eta'|$, the latter inequality follows immediately from the choice of $\delta$.

Finally, given $u\in\mathring H^\sigma(\Omega)$ with support in $\overline B$ we define $\tilde u \in\mathring H^\sigma(\R^n_+)$ by
$$
\tilde u(\xi) = u(\Phi^{-1}(\xi)) \qquad\text{if}\ \xi\in\Phi(\Omega\cap B^*)
$$
and $\tilde u(\xi) = 0$ if $\xi\in\R^n_+\setminus\Phi(\Omega\cap B^*)$. Then
\begin{align*}
I_{n,\sigma,\Omega}[u] & \geq \iint_{(\Omega\cap B^*)\times(\Omega\cap B^*)} \frac{(u(x)-u(y))^2}{|x-y|^{n+2\sigma}}\,\ud x\,\ud y \\
& = \iint_{\Phi(\Omega\cap B^*)\times\Phi(\Omega\cap B^*)} \frac{(\tilde u(\xi)-\tilde u(\eta))^2}{|\Phi^{-1}(\xi) -\Phi^{-1}(\eta)|^{n+2\sigma}}\,\ud \xi\,\ud \eta \\
& \geq (1-\epsilon) \iint_{\Phi(\Omega\cap B^*)\times\Phi(\Omega\cap B^*)} \frac{(\tilde u(\xi) -\tilde u(\eta))^2}{|\xi-\eta|^{n+2\sigma}} \,\ud \xi\,\ud \eta \\
& = (1-\epsilon) \iint_{\R^n_+\times\R^n_+} \frac{(\tilde u(\xi) -\tilde u(\eta))^2}{|\xi-\eta|^{n+2\sigma}} \,\ud \xi\,\ud \eta \\
& \qquad - 2 \int_{\Phi(\Omega\cap B^*)} \tilde u(\xi)^2 \int_{\R^n_+\setminus \Phi(\Omega\cap B^*)} \frac{\ud \eta}{|\xi-\eta|^{n+2\sigma}} \,\ud \xi \\
& \geq (1-\epsilon) I_{n,\sigma,\R^n_+}[\tilde u] - C \int_\Omega u^2\,\ud x \,.
\end{align*}
In the last inequality we use the fact that $\tilde u$ has support in $\Phi(\Omega\cap B)$ and that
$$
2 \int_{\R^n_+\setminus \Phi(\Omega\cap B^*)} \frac{\ud \eta}{|\xi-\eta|^{n+2\sigma}} \leq C
\qquad \text{for all}\ \xi\in\Phi(\Omega\cap B) \,.
$$
This proves the first part of the lemma.

For the proof of the second part, given $v\in\mathring H^\sigma(\R^n_+)$ with support in $\Phi(\Omega\cap B)$, we define $\tilde v\in\mathring H^\sigma(\Omega)$ by
$$
\tilde v(x) = v(\Phi(x)) \qquad\text{if}\ x\in\Omega\cap B^*
$$
and $\tilde v(x)=0$ if $x\in\Omega\setminus B^*$. Then
\begin{align*}
I_{n,\sigma,\Omega}[\tilde v] & = \iint_{(\Omega\cap B^*)\times(\Omega\cap B^*)} \frac{(\tilde v(x)-\tilde v(y))^2}{|x-y|^{n+2\sigma}}\,\ud x\,\ud y + 2 \int_{\Omega\cap B^*} \tilde v(x)^2 \int_{\Omega\setminus B^*} \frac{\ud y}{|x-y|^{n+2\sigma}}\,\ud x \\
& \leq \iint_{\Phi(\Omega\cap B^*)\times\Phi(\Omega\cap B^*)} \frac{(v(\xi)-v(\eta))^2}{|\Phi^{-1}(\xi)-\Phi^{-1}(\eta)|^{n+2\sigma}}\,\ud\xi\,\ud\xi + C \int_{\R^n_+} v(\xi)^2 \,\ud \xi \\
& \leq (1+\epsilon) \iint_{\Phi(\Omega\cap B^*)\times\Phi(\Omega\cap B^*)} \frac{(v(\xi)-v(\eta))^2}{|\xi-\eta|^{n+2\sigma}}\,\ud\xi\,\ud\xi + C \int_{\R^n_+} v(\xi)^2 \,\ud \xi \\
& \leq (1+\epsilon) I_{n,\sigma,\R^n_+}[v] + C \int_{\R^n_+} v(\xi)^2 \,\ud \xi \,.
\end{align*}
In the first inequality we used the fact that
$$
2 \int_{\Omega\setminus B^*} \frac{\ud y}{|x-y|^{n+2\sigma}} \leq C
\qquad\text{for all}\ x\in\Omega\cap B \,.
$$
After replacing $\delta$ by $\delta'$ such that $\Phi(\Omega\cap B)$ contains the ball with radius $\delta'$ centered at the origin, we obtain the lemma.
\end{proof}

\begin{proof}[Proof of Proposition \ref{almostineq}]
Given $\epsilon>0$ let $\delta>0$ be as in the first part of Lemma \ref{changeofvar}. We cover the boundary by finitely many balls of radius $\delta$ and choose real Lipschitz functions $\chi_0,\ldots,\chi_N$ such that
$$
\chi_0^2 + \chi_1^2 + \ldots + \chi_N^2 \equiv 1
\qquad\text{in}\ \Omega
$$
and, for $j=1,\ldots, N$, $\chi_j$ has support in one of the balls from the covering and the support of $\chi_0$ is contained in $\Omega$. A simple computation shows that
$$
I_{n,\sigma,\Omega}[u] = \sum_{j=0}^N I_{n,\sigma,\Omega}[\chi_j u] - \sum_{j=0}^N \iint_{\Omega\times\Omega} u(x) \frac{(\chi_j(x)-\chi_j(y))^2}{|x-y|^{n+2\sigma}} u(y) \,\ud x\,\ud y \,.
$$
Since the $\chi_j$ are Lipschitz, so that the singularity of the integral kernel is mitigated, it follows from the Schur test that there is a $C<\infty$ such that
$$
\sum_{j=0}^N \iint_{\Omega\times\Omega} u(x) \frac{(\chi_j(x)-\chi_j(y))^2}{|x-y|^{n+2\sigma}} u(y) \,\ud x\,\ud y \leq C \int_\Omega u^2\,\ud x \,.
$$
Let $j=1,\ldots,N$. Since the function $u_j =\chi_j u$ is supported in a ball of radius $\delta>0$ we can apply Lemma \ref{changeofvar} and we obtain
\begin{align*}
I_{n,\sigma,\Omega}[\chi_j u] & \geq (1-\epsilon) I_{n,\sigma,\R^n_+}[\tilde u_j] - C'  \int_\Omega \chi_j^2 u^2\,\ud x \\
& \geq (1-\epsilon) S_{n,\sigma}(\R^n_+) \left( \int_{\R^n_+} |\tilde u_j|^\frac{2n}{n-2\sigma}\,\ud \xi \right)^\frac{n-2\sigma}{n} 
-C'  \int_\Omega \chi_j^2 u^2\,\ud x \\
& = (1-\epsilon) S_{n,\sigma}(\R^n_+) \left( \int_{\Omega} |\chi_j u|^\frac{2n}{n-2\sigma}\,\ud x \right)^\frac{n-2\sigma}{n} -C'  \int_\Omega \chi_j^2 u^2\,\ud x\,.
\end{align*}
In the last identity we used the fact that the change of variables has Jacobian equal to one.

Finally, for $j=0$ we write
$$
I_{n,\sigma,\Omega}[\chi_0 u] = \iint_{\R^n\times\R^n} \frac{(\chi_0(x) u(x) - \chi_0(y)u(y))^2}{|x-y|^{n+2\sigma}}\,\ud x\,\ud y + \int_{\Omega} V_\Omega(x) \chi_0(x)^2 u(x)^2
$$
with $V_\Omega$ from \eqref{eq:potential}. Since $\chi_0$ is supported away from the boundary, there is a $C'<\infty$ such that
$$
V_\Omega(x) \chi_0(x)^2 \geq -C''
\qquad\text{for all}\ x\in\Omega \,.
$$
Thus,
$$
I_{n,\sigma,\Omega}[\chi_0 u] \geq S_{n,\sigma}(\R^n) \left( \int_{\Omega} |\chi_0 u|^\frac{2n}{n-2\sigma}\,\ud x \right)^\frac{n-2\sigma}{n} - C'' \int_\Omega \chi_0^2 u^2\,\ud x \,.
$$

Since, by Lemma \ref{bindingsimple}, $S_{n,\sigma}(\R^n) \geq S_{n,\sigma}(\R^n_+)$, we conclude that
$$
I_{n,\sigma,\Omega}[u] \geq (1-\epsilon) S_{n,\sigma}(\R^n_+) \sum_{j=0}^N \left( \int_\Omega |\chi_j u|^\frac{2n}{n-2\sigma}\,\ud x \right)^\frac{n-2\sigma}{n} - (C+ \max\{C',C''\}) \int_\Omega u^2\,\ud x \,.
$$
Since
\begin{align*}
\sum_{j=0}^N \left( \int_\Omega |\chi_j u|^\frac{2n}{n-2\sigma}\,\ud x \right)^\frac{n-2\sigma}{n} & = \sum_{j=0}^N \left\| \chi_j^2 u^2 \right\|_\frac{n}{n-2\sigma} \geq \left\| \sum_{j=0}^N \chi_j^2 u^2 \right\|_\frac{n}{n-2\sigma}
 = \left( \int_\Omega | u|^\frac{2n}{n-2\sigma}\,\ud x \right)^\frac{n-2\sigma}{n},
\end{align*}
we have shown the proposition.
\end{proof}

Finally, we sketch the

\begin{proof}[Proof of Proposition \ref{thresholdbdd}. Part 2]
We show that $S^{**}_{n,\sigma}(\Omega)\leq S_{n,\sigma}(\R^n)$. We may assume that $S^{**}_{n,\sigma}(\Omega)>0$, for otherwise there is nothing to show. After a translation and rotation we may assume that $0\in\partial\Omega$ and that the outward normal to $\partial\Omega$ at $0$ is $(0,\ldots,0,-1)$. Let $\epsilon\in (0,1)$ and choose $\delta>0$ as in the second part of Lemma \ref{changeofvar}. Let $0\not\equiv v\in\mathring H^\sigma(\R^n_+)$ have compact support and set $v_\lambda(x):=\lambda^\frac{n-2\sigma}{2}v(\lambda x)$.  Then for all sufficiently large $\lambda$ (depending on $\delta$) there is a $\tilde v_\lambda\in\mathring H^\sigma(\Omega)$ such that
$$
I_{n,\sigma,\Omega}[\tilde v_\lambda] \leq (1+\epsilon) I_{n,\sigma,\R^n_+}[v_\lambda] + C \int_{\R^n} v_\lambda^2\,\ud x = (1+\epsilon) I_{n,\sigma,\R^n_+}[v]+ C \lambda^{-2\sigma} \int_{\R^n_+} v^2\,\ud x
$$
and similarly $I_{n,\sigma,\Omega}[\tilde v_\lambda] \geq (1-\epsilon) I_{n,\sigma,\R^n_+}[v]- C \lambda^{-2\sigma} \int_{\R^n_+} v^2\,\ud x$. Since the change of variables has Jacobian equal to one, we have
$$
\int_\Omega |\tilde v_\lambda|^\frac{2n}{n-2\sigma}\,\ud x = \int_{\R^n_+} |v_\lambda|^\frac{2n}{n-2\sigma}\,\ud x = \int_{\R^n_+} |v|^\frac{2n}{n-2\sigma}\,\ud x \,.
$$
We clearly have $\tilde v_\lambda\rightharpoonup 0$ in $\mathring H^\sigma(\Omega)$ and, because of the lower bound on $I_{n,\sigma,\Omega}[\tilde v_\lambda]$, also $\tilde v_\lambda/I_{n,\sigma,\Omega}[\tilde v_\lambda]^{1/2}\rightharpoonup 0$ in $\mathring H^\sigma(\Omega)$. Using this sequence in the definition of $S^{**}_{n,\sigma}(\Omega)$ we obtain
$$
S^{**}_{n,\sigma}(\Omega) \leq \liminf_{\lambda\to\infty} \frac{I_{n,\sigma,\Omega}[\tilde v_\lambda]}{\left(\int_\Omega |\tilde v_\lambda|^\frac{2n}{n-2\sigma}\,\ud x \right)^\frac{n-2\sigma}{n}} \leq
(1+\epsilon) \frac{I_{n,\sigma,\R^n_+}[v]}{\left(\int_{\R^n_+} |v|^\frac{2n}{n-2\sigma}\,\ud x \right)^\frac{n-2\sigma}{n}} \,.
$$
Taking the infimum over all compactly supported $v\in\mathring H^\sigma(\R^n_+)$, which is a dense set, and recalling that $\epsilon>0$ is arbitrary we obtain the claimed inequality.
\end{proof}


\appendix

\section{Some facts about the Sobolev spaces $\mathring H^\sigma(\R^n)$}

We begin with an improvement of the fractional Sobolev inequality due to G\'erard, Meyer and Oru \cite{GeMeOr}. More general ones can be found in \cite{LiD}. For the sake of completeness we include a simple proof following the lines of \cite[Theorem 1.43]{BaChDa}. We use the notation
\[
\left( e^{t\Delta}u \right) (x):=\int_{\R^n} \frac{1}{(4\pi t)^{n/2}}e^{-|x-y|^2/4t}u(y)\,d y.
\]

\begin{lemma}\label{lem:improvedinequality}
Let $0<\sigma<1/2$ if $n=1$ and $0<\sigma<1$ if $n\geq 2$. Then there is a constant $C_{n,\sigma}$ such that for all $u\in \mathring H^\sigma(\R^n)$,
\[
\|u\|_{L^\frac{2n}{n-2\sigma}(\R^n)}\le C_{n,\sigma} \left( \int_{\R^n}\int_{\R^n} \frac{(u(x)-u(y))^2}{|x-y|^{n+2\sigma}}\,d x\,d y\right)^{\frac{n-2\sigma}{2n}} \left(\sup_{t>0}t^{\frac{n-2\sigma}{4}} \|e^{t\Delta}u\|_{L^\infty(\R^n)}\right)^{\frac{2\sigma}{n}} \,.
\]
\end{lemma}

\begin{proof} 
We abbreviate $q=\frac{2n}{n-2\sigma}$ and $\alpha=-\frac{n-2\sigma}{4}$ and assume, without loss of generality, that
 \[
 \sup_{t>0}t^{-\alpha} \|e^{t\Delta}u\|_{L^\infty(\R^n)}=1.
 \]
Note that
\[
 \int_{\R^n}|u|^q=q \int_{0}^\infty|\{|u|>\lambda\}|\lambda^{q-1}\,d\lambda.
 \]
We bound
 \[
|\{|u|>\lambda\}|\leq |\{|e^{t\Delta}u|>\lambda/2\}|+ |\{|e^{t\Delta}u-u|>\lambda/2\}|
\]
and choose $t=(\lambda/2)^{1/\alpha}$, so that the first term on the right hand side is zero. Thus, by Plancherel's theorem
\begin{align*}
|\{|u|>\lambda\}|&\leq |\{|e^{(\lambda/2)^{1/\alpha}\Delta}u-u|>\lambda/2\}|\\
&\leq (2/\lambda)^2\int_{\R^n }|e^{(\lambda/2)^{1/\alpha}\Delta}u-u|^2\,\,d x\\
&= (2/\lambda)^2\int_{\R^n }(e^{-(\lambda/2)^{1/\alpha}|\xi|^2}-1)^2|\hat u(\xi)|^2\,\,d \xi \,.
\end{align*}
Therefore,
\begin{align*}
 \int_{\R^n}|u|^q&\le q \int_{\R^n}|\hat u(\xi)|^2\,\,d \xi \int_{0}^\infty (2/\lambda)^2(e^{-(\lambda/2)^{1/\alpha}|\xi|^2}-1)^2\lambda^{q-1}\,d\lambda\\
 &= C\int_{\R^n}|\hat u(\xi)|^2|\xi|^{2\sigma}\,\,d \xi
\end{align*}
with
$$
C = q \int_{0}^\infty (2/\lambda)^2(e^{-(\lambda/2)^{1/\alpha}}-1)^2\lambda^{q-1}\,\,d\lambda = \alpha q 2^q \int_0^\infty (e^{-\mu} -1)^2 \mu^{\alpha(q-2)-1}\,d\mu <\infty \,.
$$
Another application of Plancherel's theorem concludes the proof of the inequality.
 \end{proof}

The following lemma shows that on domains of finite measure with sufficiently regular boundary there is no Sobolev inequality for $\sigma<1/2$. The proof uses ideas from \cite{Dy}.

\begin{lemma}\label{nosob}
Let $n\geq 1$, $0<\sigma<1/2$ and let $\Omega\subset\R^n$ be an open set of finite measure such that
$$
\left| \left\{ x\in\Omega:\ \dist(x,\Omega^c)<\delta \right\} \right| = o(\delta^{2\sigma}) 
\qquad\text{as}\ \delta\to 0 \,.
$$
Then
$$
\inf_{0\not\equiv u \in C^1_c(\Omega)} \frac{I_{n,\sigma,\Omega}[u]}{\left( \int_{\Omega} |u|^\frac{2n}{n-2\sigma}\,\ud x \right)^\frac{n-2\sigma}{n}} = 0 \,.
$$
\end{lemma}

Note that, if $\Omega$ is bounded Lipschitz, then $\left| \left\{ x\in\Omega:\ \dist(x,\Omega^c)<\delta \right\} \right| \lesssim \delta$ for $\delta$ sufficiently small and therefore $S_{n,\sigma}(\Omega)=0$ for $\sigma<1/2$.

\begin{proof}
Let $u_\delta\in C^1_c(\Omega)$ such that $u_\delta(x)=1$ if $x\in\Omega$ and $\dist(x,\Omega^c)\geq\delta$, $0\leq u_\delta\leq 1$ and $|\nabla u_\delta|\lesssim \delta^{-1}$. Then
\begin{align*}
I_{n,\sigma,\Omega}[u_\delta] & \leq 2 \int_{\{x\in\Omega:\ \dist(x,\Omega^c)<\delta\}} \int_\Omega \frac{(u_\delta(x)-u_\delta(y))^2}{|x-y|^{n+2\sigma}}\,\ud y\,\ud x \\
& = 2\int_{\{x\in\Omega:\ \dist(x,\Omega^c)<\delta\}} \left( I(x)+ II(x) \right)\,\ud x
\end{align*}
where
\begin{align*}
I(x) & := \int_{\{y\in\Omega:\ |x-y|<\delta\}} \frac{(u_\delta(x)-u_\delta(y))^2}{|x-y|^{n+2\sigma}}\,\ud y \\
& \lesssim \frac{1}{\delta^2} \int_{\{y\in\Omega:\ |x-y|<\delta\}} \frac{\ud y}{|x-y|^{n+2\sigma-2 }} \\
& \lesssim \frac{1}{\delta^{2\sigma}}
\end{align*}
and
\begin{align*}
II(x) & := \int_{\{y\in\Omega:\ |x-y|\geq \delta\}} \frac{(u_\delta(x)-u_\delta(y))^2}{|x-y|^{n+2\sigma}}\,\ud y \\
& \leq \int_{\{y\in\Omega:\ |x-y|<\delta\}} \frac{\ud y}{|x-y|^{n+2\sigma}} \\
& \lesssim \frac{1}{\delta^{2\sigma}} \,.
\end{align*}
Thus,
$$
I_{n,\sigma,\Omega}[u_\delta] \lesssim \delta^{-2\sigma} \left| \left\{ x\in\Omega:\ \dist(x,\Omega^c)<\delta \right\} \right| \to 0
\qquad\text{as}\ \delta\to 0 \,.
$$
Since $\int_\Omega u_\delta^\frac{2n}{n-2\sigma}\,\ud x \to |\Omega|$, we obtain the lemma.
\end{proof}



\bibliographystyle{amsalpha}

\end{document}